\documentclass{amsart}
\newtheorem{theorem}{Theorem}[section]

\usepackage{latexsym}
\usepackage{amssymb,amsmath,amsfonts}
\usepackage[colorlinks=true]{hyperref}
\usepackage{graphicx}

\usepackage{adjustbox}
\usepackage{longtable} 
\usepackage {setspace}
\usepackage{multicol, multirow}
\usepackage{bigstrut}
\usepackage{tabularx}
\usepackage{amsmath}
\usepackage{amssymb}
\usepackage{amsthm}
\usepackage{float}
\usepackage{amsfonts}
\usepackage{booktabs}
\newtheorem{lemma}[theorem]{Lemma}

\newtheorem{assup}[theorem]{Assumptions}
\newtheorem{algo}[theorem]{Algorithm}

\theoremstyle{definition}
\newtheorem{definition}[theorem]{Definition}

\newtheorem{exm}[theorem]{Example}

\theoremstyle{remark}
\newtheorem{remark}[theorem]{Remark}

\numberwithin{equation}{section}
\global\parskip 6pt \oddsidemargin
0.1cm \evensidemargin 0.4cm \headheight 0.1cm \topmargin -2.0cm
\title[Equilibrium Problem with New Version of Inertial]
{Solving Equilibrium Problem with New Inertial Technique}
\author[C. E. Nwakpa et al.]{Chidi Elijah Nwakpa$^{1},$  Chinedu Izuchukwu$^{1},$   Chibueze CHristian Okeke$^{1}$ \\ Dilber Uzun Ozsahin$^{2,3,4}$  Abubakar Adamu$^{*4,5}$}

\address{$^{1}$School of Mathematics, University of the Witwatersrand, Private Bag 3, Johannesburg 2050, South Africa.}

\address{$^{2}$Department of Medical Diagnostic Imaging, College of Health Science, University of Sharjah, Sharjah, UAE}

\address{$^{3}$Research Institute for Medical and Health Sciences, University of Sharjah, Sharjah, UAE}

\address{$^{4}$ Operational Research Center in Healthcare, Near East University, Nicosia, TRNC}

\address{$^{5}$ School of Mathematics, Chongqing Normal University, Chongqing 400047, China}

\email{*corresponding author: A. Adamu,  abubakar.adamu@neu.edu.tr}

\begin{document}
	\keywords{ Equilibrium problems; Hilbert space; self-adaptive stepsize; inertial term; correction terms; pseudomonotone; bifunction; linear convergence rate.\\
	{\rm 2000} {\it Mathematics Subject Classification}: 47H09; 47H10; 49J53; 90C25 \\ *corresponding author: A. Adamu, abubakar.adamu@neu.edu.tr}
	\begin{abstract}  We propose in this work a subgradient extragradient method with inertial and correction terms for solving equilibrium problems in a real Hilbert space.  We obtain that the sequence generated by our proposed method converges weakly to a point in the solutions set of the equilibrium problem when the associated bivariate function is pseudomonotone and satisfies Lipschitz conditions. Furthermore, in a case where the bifunction is strongly pseudomonotone, we establish a linear convergence rate. Lastly, through different numerical examples, we demonstrate that the incorporation of multiple correction terms significantly improves our proposed method when compared with other methods in the literature.
	\end{abstract}
	
	\maketitle
	\section{Introduction}\label{Itntro}
	\noindent Throughout this paper, we denote a real Hilbert space by $\mathcal{H}$ and its closed convex nonempty subset by $\mathcal{C}$.  
	Let $f$ : $\mathcal{C} \times \mathcal{C} \longrightarrow \mathbb{R}$ be a
	  bifunction such that $f(z,z) = 0$ for each $z \in \mathcal{C}$. With this bifunction $f$, an equilibrium problem (EP, in short) is defined by: 
	\begin{eqnarray}\label{1.1}
    \begin{cases}
        \mbox{obtain} \ \bar{v} \in \mathcal{C} \ \ \mbox{for \ which} \\
         f(\bar{v}, z) \geq 0, \ \ \forall z\in \mathcal{C}.
    \end{cases}
	\end{eqnarray}	
	
	\noindent Let us designate the solutions set of problem (\ref{1.1}) by $EP(f , \mathcal{C}) := \{\bar{v}\in\mathcal{C} ~\vert~ f(\bar{v},z)\geq0 \ \ \forall z\in\mathcal{C}\}$. 
	Equilibrium problem is very general in the sense that different important mathematical models such as optimization problems, fixed point problems, complementarity problems, saddle point problems, variational inequality problems, Nash equilibrium problems, are classified as special cases (see, \cite{BLU,MUU}). The broad applications of EPs to various aspect of engineering, decision making and optimal control in economics, financial analysis, mathematical sciences and others, have drawn the attention of several authors to introduce and study different methods ( see, for instance \cite{ANP,CS,FLA,HEU3,HEU4,KON,Pham,SAN,SCJ} ) of solving this problem \eqref{1.1} in finite and infinite spaces. It is worth noting that the proximal point method developed by Moudafi \cite{MOUD} is the most popular technique for solving the EP\eqref{1.1}, although this method cannot be used to solve pseudomonotone EP since the convergence may fail to exist even for a monotone problem. So, in order to overcome this challenge, the Korpelevich extragradient method (EGM) for solving the saddle point problem \cite{Korpelevich} was extended by Quoc \textit{et al.} in \cite{qu} to solve a class of EP\eqref{1.1} that is pseudomonotone, hence, ensuring convergence. According to them, the solutions of problem \eqref{1.1} can be estimated using the following method:
	\begin{eqnarray}\label{EGM}
		\begin{cases}
			y_n = \underset{z\in\mathcal{C}}{arg\min}\left\{\lambda{f(x_n, z)} + \dfrac{1}{2}\|z-x_n\|^2 \right\},\\
			x_{n+1} = \underset{z\in\mathcal{C}}{arg\min}\left\{\lambda{f(y_n, z)} + \dfrac{1}{2}\|z-x_n\|^2 \right\}, \ \mbox{with}\ x_0\in\mathcal{C} \ \ \mbox{and}\ \ \lambda>0.
		\end{cases}	   
	\end{eqnarray}
	However, the EGM poses another computational challenge, as it involves solving two strongly convex sub-problems within the convex set $\mathcal{C}$ at each computational step. Moreover, computation may become nearly impossible, particularly when the convex set $\mathcal{C}$ exhibits structural complexity.
	
	\noindent New iterative methods and modifications to existing ones are introduced to improve the computational efficiency of algorithms. For this reason, the EGM was modified, leading to the introduction of the subgradient extragradient method (SGEGM) for solving EP\eqref{1.1}. In the SGEGM, a suitable half-space is constructed, and a strongly convex program is solved over it, which lowers computational costs and enhances the performance of the iterative method. For some SGEGMs, see \cite{CS,Dadashi,SHE,SCJ} and the references quoted therein.
	
	\noindent  
	As new methods for solving Ep\eqref{1.1} and other optimization problems evolve, researchers are increasingly interested in exploring different techniques to achieve faster convergence of these methods. One approach is to incorporate inertial terms into these methods. The use of inertial-like algorithms can be traced back to Polyak's work \cite{Polyak} on the Heavy Ball with Friction (HBF), a second-order dynamical system. The inertial extrapolation term arose naturally from an implicit time discretization of the HBF system.  A distinctive feature of this algorithm is that it utilizes the previous two iterates to determine the next iterate. Building on this idea, Moudafi \cite{MOUD1} developed an inertial-proximal method for solving  EP\eqref{1.1} and obtained that the sequence generated by his method converges weakly, under certain conditions, to a point in $EP(f, \mathcal{C}).$  Motivated by the work in \cite{MOUD1}, Vinh and Muu in \cite[Algorithm 1]{NG1} proposed and obtained a weak convergence result for an inertial EGM for solving EP\eqref{1.1} by incorporating the on-line rule, which is a restrictive condition that is not always desirable in real-life problems. This is because the on-line rule can slow the convergence of an algorithm, particularly when dealing with large-scale problems. To obtain a faster-converging algorithm for solving EP\eqref{1.1}, the authors in \cite{Nwakpa}, under some implementable conditions, proposed a relaxed inertial method and established a weak convergence result and a linear convergence rate.
	
\noindent 
Furthermore, another technique for accelerating algorithms for solving optimization problems involves incorporating correction terms. On this note, Kim \cite{Kim} in his work introduced an accelerated proximal point algorithm that combines the proximal point algorithm with inertial and correction terms. Although no weak convergence result was obtained for the generated sequences, Kim \cite[Theorem 4.1]{Kim} leveraged the performance estimation problem (PEP) approach of Drori and Teboulle \cite{Drori} to obtain the worst-case convergence rate $\Big(\|y_n-w_{n-1}\|=\mathcal{O}(n^{-1})\Big)$ for the following algorithm: 
\begin{eqnarray}
	\begin{cases}
		w_n=y_n+\dfrac{n-1}{n+1}(y_n-y_{n-1})+\dfrac{n-1}{n+1}(w_{n-2}-y_{n-1}),\\
		\\
		y_{n+1}=J^A_\lambda(w_n).
	\end{cases}
\end{eqnarray}
In another related work, Maingé \cite{Maingé} further investigated the proximal point algorithm by combining inertial, relaxation, and correction terms, with a focus on solving monotone inclusion problems. He proposed the following method
\begin{eqnarray}
	\begin{cases}
		w_n=y_n+\alpha_n(y_n-y_{n-1})+\delta_n(w_{n-1}-y_n),\\
		\\
		y_{n+1}=\dfrac{1}{1+\alpha_n}w_n+\dfrac{\alpha_n}{1+\alpha_n}J^A_{\lambda(1+\alpha_n)}(w_n),
	\end{cases}
\end{eqnarray}
and established weak convergence results and the fast rate, $\|y_{n+1}-y_n\|=o(n^{-1}),$ for the generated sequence. However, no linear convergence result was obtained. Motivated by the works of Kim \cite{Kim} and Maingé \cite{Maingé}, Izuchukwu \textit{et al.} \cite{Izuchukwu} proposed another method for solving proximal point problems. Their method involves incorporating two correction terms, yielding an approach that can be described as:
\begin{eqnarray}
	\begin{cases}
		w_n=y_n+\alpha(y_n-y_{n-1})+\delta(1+\alpha)(w_{n-1}-y_n)-\alpha\delta(w_{n-2}-y_{n-1}),\\
		\\
		y_{n+1} = J^A_\lambda(w_n).
	\end{cases}
\end{eqnarray}
They obtained, in \cite[Theorem 3.5, Theorem 3.7]{Izuchukwu}, a weak convergence result and linear convergence rate, respectively, for their proposed method under certain conditions. Through numerical examples, they further revealed that the incorporation of multiple correction terms significantly accelerates their method, outperforming those in \cite{Kim,Maingé} that rely on a single correction term.

\noindent Building on the works of Izuchukwu \textit{et al.} \cite{Izuchukwu}, Kim \cite{Kim}, and Maingé \cite{Maingé}, this paper proposes a subgradient extragradient method with inertial and correction terms for solving equilibrium problems, as formulated in \eqref{1.1}. The key contributions of this approach include:
\begin{itemize}
	\item Combining one inertial term with two correction terms to propose a subgradient extragradient method for solving pseudomonotone equilibrium problems in a real Hilbert space;
	\item  Incorporating a self-adaptive stepsize, which is distinct from the methods proposed in \cite{Izuchukwu,Kim,Maingé} and the constant stepsize used in \cite{VIN};
	\item Obtaining a weak convergence result and a linear convergence rate for the sequence generated by our proposed method, and providing numerical examples to justify the effectiveness of our proposed method.
\end{itemize}
		
	 \noindent We organize the rest of this paper as follows: 
	In Section \ref{pre}, we collect some basic definitions and technical lemmas needed in subsequent sections. In Section \ref{MAIN RESULTS}, we present and discuss our proposed method. Also, a weak convergence result is given in this section.  In Section \ref{Convergence Rate}, we obtain a linear convergence rate of our proposed method, while Section \ref{numerics} provides numerical results that assess the performance of our proposed method in comparison to other established methods in the literature.
	
	\hfill
	
	\section{Preliminaries}\label{pre}
	\noindent This section establishes the necessary foundation for the convergence analysis of our proposed algorithm by providing fundamental definitions and lemmas. For our analysis, we will use the symbols $\to$  and $\rightharpoonup$ to denote strong and weak convergence, respectively.\\

	\begin{definition}\noindent\\
		Let $\mathcal{C}$ be a nonempty, closed, and convex subset of a real Hilbert space $\mathcal{H}.$  
		For each $s, w,z \in\mathcal{C},$ a bifunction $f:\mathcal{C} \times \mathcal{C} \longrightarrow \mathbb{R}$ is said to be:
		\begin{enumerate}
			\item[i.] monotone on $\mathcal{C}$ if
			\begin{eqnarray*}
				 f(s,w) + f(w,s)\leq0;
			\end{eqnarray*}
			\item[ii.] strongly pseudomonotone on $\mathcal{C}$ if there exists a positive constant $\beta$ such that
			\begin{eqnarray*}
				f(s,w)\geq0 \implies f(w,s) +\beta\|s-w\|^2 \leq0;
			\end{eqnarray*}
			\item[iii.] pseudomonotone on $\mathcal{C}$ if
			\begin{eqnarray*}
				f(s,w)\geq0 \implies f(w,s)\leq0;
			\end{eqnarray*}
			\item[iv.] satisfying Lipschitz-type conditions on $\mathcal{C}$ if there exist two constants $\delta_1>0$ and $\delta_2>0$ such that
			\begin{eqnarray*}
				f(s,z)-f(s,w)-f(w,z)\leq \kappa_1\|s-w\|^2+\kappa_2\|w-z\|^2.
			\end{eqnarray*}
		\end{enumerate}
		\noindent Clearly, $(i)\implies (iii)$ and $(ii)\implies(iii)$ but the converse statements are not generally true.
	\end{definition}
	\noindent In the following definition, we will recall the concept of proximity operator originally introduced by Moreau \cite{Moreau}.
	\begin{definition}\noindent\\
		Let $g:\mathcal{C}\subset\mathcal{H}\longrightarrow\mathbb{R}\cup\{+\infty\}$ be a function that satisfies properness, convexity and lower semicontinuity, and $\lambda>0.$ The Moreau envelope of $g$ is the function
		\begin{eqnarray*}
			^\lambda g(s) = \underset{z\in\mathcal{C}}{\inf}\left\{g(z)+\dfrac{1}{2\lambda}\|z-s\|^2, \ \ \forall s\in\mathcal{H}\right\}.
		\end{eqnarray*}
		But for each $s\in\mathcal{H},$ the function
		\begin{eqnarray*}
			z\longmapsto g(z)+\dfrac{1}{2\lambda}\|z-s\|^2
		\end{eqnarray*}
		is strongly convex and satisfies the properness and lower semicontinuity properties. Thus, it attains its infimum in $\mathbb{R}.$ The unique minimizer of the function $z\longmapsto g(z)+\dfrac{1}{2\lambda}\|z-s\|^2,$ denoted by $prox_g(s),$ is known as the proximal point operator of $g$ at $s.$ Hence, the operator $prox_{\lambda g}:\mathcal{H}\longrightarrow\mathbb{R}$ is defined by 
		\begin{eqnarray*}
			prox_{\lambda g}(s) := \underset{z\in\mathcal{C}}{arg\min}\left\{\lambda g(z)+\dfrac{1}{2}\|s-z\|^2\right\}.
		\end{eqnarray*}
		If $g=\delta_\mathcal{C},$ where $\delta_\mathcal{C}$ is the indicator function of the set $\mathcal{C},$ then
		\begin{eqnarray*}
			prox_{\delta_\mathcal{C}}(s) = \mathcal{P}_\mathcal{C}(s) \ \ \forall s\in\mathcal{H}.
		\end{eqnarray*}
	\end{definition}
	\begin{lemma}\cite{BAUS}\label{A1}
		For each $s\in\mathcal{H},$ $z\in\mathcal{C}$ and positive constant $\lambda,$ the following inequality is true\\ \\
		$\lambda(g(z)-g(prox_{\lambda g}(s)))\geq \langle s- prox_{\lambda g}(s), z-prox_{\lambda g}(s)\rangle.$
	\end{lemma}
\begin{definition}\noindent
	\begin{enumerate}
		\item[i.] Let $f:\mathcal{H}\times \mathcal{H}\longrightarrow\mathbb{R}$ be a convex bifunction, then, for each $s,w\in \mathcal{H},$  the subdifferential of  $f(s,.)$ at $w$ is defined by
		\begin{eqnarray*}
		   \partial_2 f(s,w)&:=& \{u\in\mathcal{H} : f(s,z)\geq f(s,w)+ \langle u, z-w\rangle, \ \ \forall z\in\mathcal{H}\}.
		\end{eqnarray*}
		And in particular,
			\begin{eqnarray*}
		   \partial_2f(s,s)&:=& \{u\in\mathcal{H} : f(s,z)\geq \langle u, z-s\rangle, \ \ \forall z\in\mathcal{H}\}.
		\end{eqnarray*}
		For any $s$ fixed in $\mathcal{H},$ the bifunction $f$ is said to be  subdifferentiable at $w$ if $\partial_2f(s,w)\neq\emptyset.$ It is said to be subdifferentiable on a set $\mathcal{C}$ in $\mathcal{H}$ if it is subdifferentiable at every point $w\in\mathcal{C},$ and if $dom(\partial_2 f)=\mathcal{H},$
		  then, we say that $f$ is subdifferentiable.
		\item[ii.] The normal cone $\mathcal{N}_C$ of the convex set $\mathcal{C}$ is defined at a point $s$ by:
		\begin{eqnarray*}
			\mathcal{N}_C(s)= \{v\in\mathcal{H} : \langle v, z-s\rangle \leq 0, \ \ \forall z\in\mathcal{C}\}.
		\end{eqnarray*}
		\item[iii.] Let $s\in\mathcal{H},$ there exists a unique closest point in $\mathcal{H}$ denoted by $\mathcal{P}_\mathcal{C}(s)$ such that
		\begin{eqnarray*}
			\mathcal{P}_C(s) = arg\min\{\|s-z\| \ \  :~z\in\mathcal{C} \}.
		\end{eqnarray*}
		 $\mathcal{P}_C$ is called the metric projection of $\mathcal{H}$ onto $\mathcal{C}.$
	\end{enumerate}
\end{definition}
 \begin{remark}\label{subdiff Nc}\noindent
	\begin{enumerate}
		\item [i.] Notice that for any convex function $f:\mathcal{H}\times \mathcal{H}\longrightarrow\mathbb{R}$ with $s$ fixed in $\mathcal{H}$ and $w\in int(dom~f),$ the subdifferential $\partial_2 f(s,w)$ (set of subgradients) is nonempty and compact in $\mathcal{H}$ \cite[Propostiton 2.47]{Mordukhovich}, \cite{Rock}.
		\item [ii.] For any point $s\in\mathcal{C},$ the normal cone $\mathcal{N}_\mathcal{C}$ is nonempty since it always contains the origin \cite{Mordukhovich}.
	\end{enumerate}
\end{remark}

	\begin{lemma}\cite{Rock}\label{optimal lemma}
		Let $g:\mathcal{C}\longrightarrow\mathbb{R}$ be a convex, lower semicontinuous and subdifferentiable function on $\mathcal{C}$. Suppose $int(\mathcal{C}) \neq 0,$ then $\bar{r}$ is a solution of the convex minimization problem, $\min\{g(s):s\in\mathcal{C}\}$ if and only if $\partial g(\bar{r})+\mathcal{N}_\mathcal{C}( \bar{r})\ni0.$
	\end{lemma}
\begin{lemma}\label{C1}
	The following is true for each $s,z\in \mathcal{H}$ and $\beta \in \mathbb{R};$
	\begin{itemize}
		\item [(i)] 
			$\|(1-\beta) s -\beta z\|^2 = (1-\beta)\|s\|^2 - \beta\|z\|^2 + \beta(1-\beta)\|s-z\|^2;$
		\\
		\item [(ii)] 
			$2\langle s,z\rangle = \|s\|^2+\|z\|^2-\|s-z\|^2= \|s+z\|^2-\|s\|^2-\|z\|^2.$
		\end{itemize}
\end{lemma}
\begin{lemma}\cite{Opial} \label{D1}
Let $\mathcal{C}\subset \mathcal{H}$ be nonempty, and $\{x_n\}$ a sequence in $\mathcal{H}$ such that:
\begin{itemize}
	\item[(i)]  $\underset{n\to \infty}\lim\|x_n - \bar{r}\|$ exists for each $\bar{r}\in \mathcal{C}$;
	\item[(ii)] every sequentially weak cluster point of $\{x_n\}$ is found in $\mathcal{C}.$\\ \\ Then the sequence $\{x_n\}$ weakly converges to a point in $\mathcal{C}.$
\end{itemize} 
	 \begin{definition}
	A sequence $\{x_n\}$ in $\mathcal{H}$ is said to converge $R-$linearly to a point $\bar{v}$ with rate $\rho\in[0,1)$ if there is a constant $\kappa>0$ such that
	$$\|x_n-\bar{v}\|\leq\kappa\rho^n \ \ \ \mbox{for all} \ n\in\mathbb{N}.$$
	\end{definition}
\end{lemma}
	\section{Main Results}\label{MAIN RESULTS}
We will assume that the bifunction $f$ satisfies the following conditions:
\begin{assup}\label{B}
	\ \
\begin{itemize}
	\item[(B1)] For each $s\in\mathcal{C}, \ \ f(s,s) = 0$ and $f$ is pseudomonotone on $\mathcal{C};$
	 
	\item[(B2)] With two positive constants $\kappa_1$ and $\kappa_2,$ $f$ satisfies the Lipschitz-type conditions on $\mathcal{H};$
	
	\item[(B3)] $f(s,.)$ is convex, lower semicontinuous and subdifferentiable on $\mathcal{C}$ for each $s\in\mathcal{C};$
	
	\item[(B4)] For all fixed point $y\in\mathcal{C},$ $f(., y)$ is sequentially weakly upper semicontinuous on $\mathcal{C}.$ That is, if $x_n\in\mathcal{C}$ is such that $x_n\rightharpoonup x^*,$   as $n\to\infty,$ then, $\underset{n\to\infty}{\limsup}f(x_n, y_)\leq f(x^*,y)$ for each fixed $y\in\mathcal{C}$ and $x_n\in\mathcal{C};$
	
	\item[(B5)] $EP(f,\mathcal{C})\neq\emptyset.$
\end{itemize}
\end{assup}
\begin{assup}\label{ASSUM EP}
	Suppose that $\alpha\in[0,1)$ and $\delta\in(0,1)$ satisfy the following conditions:

	$$\alpha\leq\dfrac{1}{2} \; \; \mbox{and} \; \;  \max\left\{\dfrac{2\alpha}{1+\alpha}, \dfrac{(\alpha^2+2)-\sqrt{\alpha^4-8\alpha^3-8\alpha^2+4}}{2\alpha}\right\}<\delta<1.$$
\end{assup}
Now, we introduce our proposed method as follows:

\begin{algo}\label{ALG}
\hrule\hrule
 Subgradient extragradient method with inertial and correction terms.
\hrule\hrule
\noindent \textbf{STEP 0:} Choose $\alpha$ and $\delta$ such that Assumption \ref{ASSUM EP} is true. Take $w_{-1} = w_{-2}, \ y_0, \ y_{-1}\in\mathcal{H}$ arbitrarily. Let $\lambda_0>0, \ \mu\in]0,1[$ and set $n=0$.\\
\textbf{STEP 1}: Whenever $w_{n-1}, \ w_{n-2}, \ y_{n-1}$ and $y_n$ are given,
compute
\begin{eqnarray*}
	\begin{cases}
		w_n = y_n+\alpha(y_n-y_{n-1}) +\delta(1+\alpha)(w_{n-1}-y_n)-\alpha\delta(w_{n-2}-y_{n-1}),\\ \ \\
		z_n = \underset{z\in\mathcal{C}}{argmin}\left\{\lambda_nf(w_n,z)+\dfrac{1}{2}\|w_n-z\|^2\right\} = \mbox{prox}_{\lambda_n f(w_n, \cdot)}(w_n).
	\end{cases}
\end{eqnarray*}
 Stop the computation if $z_n = w_n,$ and $z_n$ becomes a solution, otherwise, move to step 2.\\
\textbf{STEP 2}:\\
\noindent Choose $v_n\in\partial_2 f(w_n, .)(z_n)$ in such a way there is $b_n\in \mathcal{N}_\mathcal{C}(z_n)$ satisfying\\
$b_n = w_n-\lambda_nv_n-z_n$ and construct a half space\\
$ T_n = \{x\in\mathcal{H}: \langle w_n-\lambda_nv_n-z_n, x-z_n\rangle \leq0\}.$\\

\noindent Evaluate\\
$y_{n+1} = \underset{z\in T_n}{argmin}\left\{\lambda_nf(z_n,z)+\dfrac{1}{2}\|w_n-z\|^2\right\}  = \mbox{prox}_{\lambda_n f(z_n, \cdot)}(w_n).$\\
\\
\textbf{STEP 3}: Update the stepsize as follows: \\
\begin{eqnarray*}
	\lambda_{n+1} = \begin{cases}
		\min\left\{\dfrac{\mu}{2}\dfrac{\|w_n-z_n\|^2+\|y_{n+1}-z_n\|^2}{f(w_n,y_{n+1})-f(w_n,z_n)-f(z_n,y_{n+1})},\ \ \lambda_n\right\}, \ \ \ \mbox{if} \ f(w_n,y_{n+1})-f(w_n,z_n)-f(z_n,y_{n+1})>0,\\ \\
		\lambda_n, \ \ \ \mbox{otherwise.}
	\end{cases} 
\end{eqnarray*}
 \noindent Set $n+1\leftarrow n$ and go back to \textbf{step 1}.
\end{algo}
\hrule\hrule
\ \
 \begin{remark} \noindent
	\begin{itemize}
		\item Since $z_n\in\mathcal{C}$ and $f$ is convex in $\mathcal{C}$, we have by Remark \ref{subdiff Nc} that $\partial_2 f(w_n,.)(z_n)\neq\emptyset$ and $\mathcal{N}_\mathcal{C}(z_n)\neq\emptyset.$ Therefore, there exist  $v_n\in \partial_2 f(w_n,.)(z_n)$ and $b_n\in\mathcal{N}_\mathcal{C}(z_n)$ such that step 2 of Algorithm \ref{ALG} is well defined.
		\item If Algorithm \ref{ALG} terminates after a finite number of iterations, that is, $z_n =w_n,$ then $z_n$ becomes a solution to the equilibrium problem. In fact, by the definition of $z_n$ and Lemma \ref{optimal lemma}, it follows that
		\begin{eqnarray*}
			0\in\partial_2\left(\lambda_nf(w_n,.)+\dfrac{1}{2}\|w_n-.\|^2\right)(z_n)+\mathcal{N}_\mathcal{C}(z_n).
		\end{eqnarray*} 
		With $v_n\in\partial_2 f(w_n, .)(z_n)$ and $b_n\in\mathcal{N}_\mathcal{C}(z_n),$ then
		 $$\lambda_nv_n+z_n-w_n+b_n = 0.$$
		Since by the assumption, $z_n=w_n,$ it then follows that $$\lambda_nv_n+b_n = 0.$$
		This further implies that 
		\begin{eqnarray}\label{1rmk}
			\lambda_n\langle v_n, z-z_n\rangle+\langle b_n, z-z_n\rangle = 0, \ \ \forall z\in\mathcal{C}.
		\end{eqnarray}
		But $b_n\in\mathcal{N}_\mathcal{C}(z_n)$ implies that $\langle b_n, z-z_n\rangle\leq0,$ $\forall z\in\mathcal{C}.$ So, we have from \eqref{1rmk} that
		\begin{eqnarray*}
			\lambda_n\langle v_n, z-z_n\rangle\geq0, \ \ \forall z\in\mathcal{C}.
		\end{eqnarray*}
		Given that $v_n\in\partial_2 f(w_n,.)(z_n)$ and $\lambda_n>0$ $\forall n\geq 1,$ then we can use the definition of subdifferential to obtain
		\begin{eqnarray*}
			0\leq\langle v_n, z-z_n\rangle\leq f(w_n,z)-f(w_n,z_n) = f(w_n,z), \ \ \forall z\in\mathcal{C}.
		\end{eqnarray*}
		Hence, $f(w_n,z)\geq 0,$ $\forall z\in\mathcal{C}.$ That is, $z_n=w_n\in EP(f, \mathcal{C}).$
		However, throughout this section, to ensure that Algorithm \ref{ALG} generates an infinite sequence of iterations, we will assume that Algorithm \ref{ALG} does not terminate in a finite number of iterations.
	\end{itemize}
\end{remark}

\hfill

         \begin{remark}\noindent\\
         	\noindent Setting $\alpha=0$ and $\delta = 0$ in our proposed Algorithm \ref{ALG} reduces it to the method in \cite{Dadashi}. The methods in \cite{HEU2,Izuchukwu-Ogwo-Zinsou,NG1,hur,REH,Shehu,VIN,Yang} incorporate an inertial extrapolation step. In contrast, our Algorithm \ref{ALG} extends the subgradient extragradient method by introducing an inertial extrapolation term, $\alpha(y_n - y_{n-1}),$ along with two correction terms, $\delta(1 +\alpha)(w_{n-1} - y_n)$ and $\alpha\delta (w_{n-2} - y_{n-1}).$ The inclusion of these terms aims to enhance the numerical convergence speed of the subgradient extragradient algorithm.
         \end{remark}

        \hfill

\begin{lemma}\cite{Shehu}\label{Lambda}
	The sequence $\{\lambda_n\}_{n\geq1}$ generated by Algorithm \ref{ALG} is a monotone non-increasing sequence with $\min\left\{\lambda_1, \dfrac{\mu}{2\max\{\kappa_1, \kappa_2\}}\right\}$ as its lower bound.
\end{lemma}

		 \hfill

\begin{lemma}\label{L2}
		Let $\{w_n\}, \ \{y_n\}$ be the sequences generated by Algorithm \ref{ALG}. Suppose that Assumption \ref{ASSUM EP} holds. Then, under conditions $(B1), \ (B2), \ (B3)$ and $(B5)$ of Assumption \ref{B}, $\{w_n\}, \ \{y_n\}$ are bounded.
	\end{lemma}
	\begin{proof}\noindent\\
	Since $z_n\in\mathcal{C},$ we obtain from Lemma \ref{optimal lemma} that
		\begin{eqnarray*}
			0\in\partial_2\left(\lambda_nf(w_n,.)+\frac{1}{2}\|.-w_n\|^2\right)(z_n)+\mathcal{N}_\mathcal{C}(z_n).
		\end{eqnarray*}
		Thus, there is a vector $v_n\in\partial_2 f(w_n, \cdot)(z_n)$ and $b_n\in\mathcal{N}_\mathcal{C}(z_n)$ such that 
		$$ w_n-\lambda_nv_n-z_n-b_n=0.$$
		From the above expression, it follows that
		\begin{eqnarray}\label{1}
			\langle w_n-z_n, z-z_n\rangle &=& \langle \lambda_nv_n+b_n, z-z_n\rangle\nonumber\\
			&=&\lambda_n\langle v_n, z-z_n\rangle +\langle b_n, z-z_n\rangle.
		\end{eqnarray}
		But $b_n\in\mathcal{N}_\mathcal{C}(z_n)$ implies that $\langle b_n, z-z_n\rangle \leq0 \ \ \ \forall z\in\mathcal{C}.$\\ 
		Thus, we obtain from \eqref{1} that
		\begin{eqnarray}\label{2}
			\langle w_n-z_n, z-z_n\rangle \leq \lambda_n\langle v_n, z-z_n\rangle \ \ \ \forall z\in \mathcal{C}.
		\end{eqnarray}
		Also, by choosing $v_n\in\partial_2 f(w_n,z_n)$ we obtain that
		\begin{eqnarray}\label{3}
			f(w_n,z)-f(w_n,z_n)\geq \langle v_n, z-z_n\rangle \ \ \forall z\in\mathcal{H}. 
		\end{eqnarray}
		But $\mathcal{C}\subset\mathcal{H},$ hence, using \eqref{2} in \eqref{3}, we have
		
		\begin{eqnarray}\label{4}
			\lambda_n(f(w_n,z)-f(w_n,z_n))\geq \langle w_n-z_n, z-z_n\rangle \ \ \forall z\in\mathcal{C}. 
		\end{eqnarray}
		Putting $y_{n+1} = z$ in \eqref{4}, then 
		\begin{eqnarray}\label{5}
			\lambda_n(f(w_n,y_{n+1})-f(w_n,z_n))\geq \langle w_n-z_n, y_{n+1}-z_n\rangle.
		\end{eqnarray}
		Furthermore, from the definition of $y_{n+1}$ and Lemma \ref{A1}, we get
		\begin{eqnarray}\label{6}
			\lambda_n(f(z_n,z)-f(z_n,y_{n+1}))\geq \langle w_n-y_{n+1}, z-y_{n+1}\rangle \ \ \forall z\in T_n.
		\end{eqnarray}
		Let $z = \bar{v}\in EP(f, \mathcal{C})$, then
		\begin{eqnarray}\label{7}
			\lambda_n(f(z_n,\bar{v})-f(z_n,y_{n+1}))\geq \langle w_n-y_{n+1}, \bar{v}-y_{n+1}\rangle \ \ \forall \bar{v}\in EP(f,\mathcal{C}).
		\end{eqnarray}
		Observe that $z_n\in\mathcal{C},$ thus, with  $\bar{v}\in EP(f,\mathcal{C}),$ then $f(\bar{v}, z_n)\geq0.$ This further implies by pseudomonotonicity of $f$ that $$f(z_n, \bar{v})\leq0.$$\\
		Hence, it ensues from \eqref{7} that
		\begin{eqnarray}\label{8}
			-\lambda_nf(z_n,y_{n+1})\geq \langle w_n-y_{n+1}, \bar{v}-y_{n+1}\rangle.
		\end{eqnarray}
		Combining \eqref{5} and \eqref{8}, we get
		\begin{eqnarray}\label{9}
			2\lambda_n(f(w_n, y_{n+1})-f(w_n, z_n)-f(z_n, y_{n+1}))\geq2\langle w_n-z_n, y_{n+1}-z_n\rangle + 2\langle w_n-y_{n+1}, \bar{v}-y_{n+1}\rangle.
		\end{eqnarray}
		Moreover,\\
		\begin{eqnarray}\label{10}
			2\langle w_n-z_n, y_{n+1}-z_n\rangle = \|w_n-z_n\|^2+\|y_{n+1}-z_n\|^2-\|y_{n+1}-w_n\|^2.
		\end{eqnarray}
		Also,\\
		\begin{eqnarray}\label{11}
			2\langle w_n-y_{n+1}, \bar{v}-y_{n+1}\rangle = \|y_{n+1}-w_n\|^2+\|y_{n+1}-\bar{v}\|^2-\|w_n-\bar{v}\|^2.
		\end{eqnarray}
		Putting \eqref{10} and \eqref{11} into \eqref{9}, gives us that\\
		\begin{eqnarray}\label{12}
			&&2\lambda_n(f(w_n, y_{n+1})-f(w_n, z_n)-f(z_n, y_{n+1}))\nonumber\\
			&\geq& \|w_n-z_n\|^2+\|y_{n+1}-z_n\|^2+\|y_{n+1}-\bar{v}\|^2-\|w_n-\bar{v}\|^2.
		\end{eqnarray}
		Using the definition of the sequence $\{\lambda_n\}_{n\geq 1},$ we get
		\begin{eqnarray}\label{13}
			2(f(w_n, y_{n+1})-f(w_n, z_n)-f(z_n, y_{n+1})) \leq \frac{\mu(\|w_n-z_n\|^2+\|y_{n+1}-z_n\|^2)}{\lambda_{n+1}}.
		\end{eqnarray}
		Using \eqref{13} in \eqref{12} we obtain\\
		\begin{eqnarray}\label{14}
			\mu\frac{\lambda_n}{\lambda_{n+1}}(\|w_n-z_n\|^2+\|y_{n+1}-z_n\|^2)\geq \|w_n-z_n\|^2+\|y_{n+1}-z_n\|^2+\|y_{n+1}-\bar{v}\|^2-\|w_n-\bar{v}\|^2.
		\end{eqnarray}
		And rearranging, we have
		\begin{eqnarray}\label{15}
			\|y_{n+1}-\bar{v}\|^2 \leq \|w_n-\bar{v}\|^2-\left(1-\mu\frac{\lambda_n}{\lambda_{n+1}}\right)\Big(\|w_n-z_n\|^2+\|y_{n+1}-z_n\|^2\Big).
		\end{eqnarray}
		Using the fact that $\|\sum_{i=1}^{n}x_i\|^2\leq n\sum_{i=1}^{n}\|x_i\|^2,$ we have
		$$-\sum_{i=1}^{2}\|x_i\|^2\leq-\dfrac{1}{2}\|\sum_{i=1}^{2}x_i\|^2.$$
		Let $x_1=w_n-z_n, \ \ x_2=z_n-y_{n+1},$ then we get from \eqref{15} that
		\begin{eqnarray}\label{15i}
			\|y_{n+1}-\bar{v}\|^2 \leq \|w_n-\bar{v}\|^2-\dfrac{1}{2}\left(1-\mu\frac{\lambda_n}{\lambda_{n+1}}\right)\|y_{n+1}-w_n\|^2.
		\end{eqnarray}
		From step 1 of Algorithm \ref{ALG}, let $u_n=y_n+\delta(w_{n-1}-y_n).$ Therefore 
		$$w_n=u_n+\alpha(u_n-u_{n-1}).$$
		Hence, it follows that
		$$y_n=\dfrac{1}{1-\delta}u_n-\dfrac{\delta}{1-\delta}w_{n-1},$$
		and
		\begin{eqnarray}\label{15b}
			y_{n+1}=\dfrac{1}{1-\delta}u_{n+1}-\dfrac{\delta}{1-\delta}w_n.
		\end{eqnarray}
		Using Lemma \ref{C1}(i) and \eqref{15b}, we get
		\begin{eqnarray}\label{16}
			\|y_{n+1}-\bar{v}\|^2 &=& \|\dfrac{1}{1-\delta}(u_{n+1}-\bar{v})-\dfrac{\delta}{1-\delta}(w_n-\bar{v})\|^2\nonumber\\
			&=&\dfrac{1}{1-\delta}\|u_{n+1}-\bar{v}\|^2-\dfrac{\delta}{1-\delta}\|w_n-\bar{v}\|^2+\dfrac{\delta}{(1-\delta)^2}\|u_{n+1}-w_n\|^2.
		\end{eqnarray}
		Putting \eqref{16} in \eqref{15}, we have
		\begin{eqnarray*}
			\dfrac{1}{1-\delta}\|u_{n+1}-\bar{v}\|^2-\dfrac{\delta}{1-\delta}\|w_n-\bar{v}\|^2+\dfrac{\delta}{(1-\delta)^2}\|u_{n+1}-w_n\|^2\leq \|w_n-\bar{v}\|^2-\dfrac{1}{2}\left(1-\mu\frac{\lambda_n}{\lambda_{n+1}}\right)\|y_{n+1}-w_n\|^2.
		\end{eqnarray*}
		That is,
			\begin{eqnarray}\label{17}
			\dfrac{1}{1-\delta}\|u_{n+1}-\bar{v}\|^2\leq\dfrac{1}{1-\delta}\|w_n-\bar{v}\|^2-\dfrac{1}{2}\left(1-\mu\frac{\lambda_n}{\lambda_{n+1}}\right)\|y_{n+1}-w_n\|^2-\dfrac{\delta}{(1-\delta)^2}\|u_{n+1}-w_n\|^2.
		\end{eqnarray}
		But 
		\begin{eqnarray}\label{18}
			\|y_{n+1}-w_n\|^2 &=& \|y_{n+1}-u_n-\alpha(u_n-u_{n-1})\|^2\nonumber\\
			&=& \|y_{n+1}-u_n\|^2-2\alpha\langle y_{n+1}-u_n, u_n-u_{n-1}\rangle +\alpha^2\|u_n-u_{n-1}\|^2\nonumber\\
			&\geq&\|y_{n+1}-u_n\|^2-2\alpha\|y_{n+1}-u_n\|\|u_n-u_{n-1}\|+\alpha^2\|u_n-u_{n-1}\|^2\nonumber\\
			&\geq&\|y_{n+1}-u_n\|^2-\alpha\Big[\|y_{n+1}-u_n\|^2 +\|u_n-u_{n-1}\|^2\Big]+\alpha^2\|u_n-u_{n-1}\|^2\nonumber\\
			&=&(1-\alpha)\|y_{n+1}-u_n\|^2-\alpha(1-\alpha)\|u_n-u_{n-1}\|^2.
		\end{eqnarray}
	Replacing $y_{n+1}$ in \eqref{18} with $u_{n+1},$ then 	
	\begin{eqnarray}\label{19}
		\|u_{n+1}-w_n\|^2\geq(1-\alpha)\|u_{n+1}-u_n\|^2-\alpha(1-\alpha)\|u_n-u_{n-1}\|^2.
	\end{eqnarray}	
		So, putting \eqref{18} and \eqref{19} in \eqref{17}, we have
			\begin{eqnarray}\label{20}
			\dfrac{1}{1-\delta}\|u_{n+1}-\bar{v}\|^2&\leq&\dfrac{1}{1-\delta}\|w_n-\bar{v}\|^2-\dfrac{1}{2}\left(1-\mu\frac{\lambda_n}{\lambda_{n+1}}\right)\Big[(1-\alpha)\|y_{n+1}-u_n\|^2-\alpha(1-\alpha)\|u_n-u_{n-1}\|^2 \Big]\nonumber\\
			&&\;-\;\dfrac{\delta}{(1-\delta)^2}\Big[(1-\alpha)\|u_{n+1}-u_n\|^2-\alpha(1-\alpha)\|u_n-u_{n-1}\|^2\Big].
		\end{eqnarray}
		Furthermore,
		\begin{eqnarray}\label{21}
			\|w_n-\bar{v}\|^2 &=& \|(1+\alpha)(u_n-\bar{v})-\alpha(u_{n-1}-\bar{v})\|^2\nonumber\\
			&=&(1+\alpha)\|u_n-\bar{v}\|^2-\alpha\|u_{n-1}-\bar{v}\|^2+\alpha(1+\alpha)\|u_n-u_{n-1}\|^2.
		\end{eqnarray}
		Applying \eqref{21} to \eqref{20} gives
		\begin{eqnarray}\label{22}
			\dfrac{1}{1-\delta}\|u_{n+1}-\bar{v}\|^2&\leq&\dfrac{1}{1-\delta}(1+\alpha)\|u_n-\bar{v}\|^2-\dfrac{\alpha}{1-\delta}\|u_{n-1}-\bar{v}\|^2+\dfrac{\alpha(1+\alpha)}{1-\delta}\|u_n-u_{n-1}\|^2\nonumber\\
			&&\;-\;\dfrac{1}{2}(1-\alpha)\left(1-\mu\dfrac{\lambda_n}{\lambda_{n+1}}\right)\|y_{n+1}-u_n\|^2+\dfrac{1}{2}\alpha(1-\alpha)\left(1-\mu\dfrac{\lambda_n}{\lambda_{n+1}}\right)\|u_n-u_{n-1}\|^2\nonumber\\
			&&\;-\;\dfrac{\delta(1-\alpha)}{(1-\delta)^2}\|u_{n+1}-u_n\|^2+\dfrac{\alpha\delta(1-\alpha)}{(1-\delta)^2}\|u_n-u_{n-1}\|^2.
		\end{eqnarray}
	Since $\underset{n\to\infty}{\lim}\lambda_n$ exists, then $\underset{n\to\infty}{\lim}\left(1-\mu\dfrac{\lambda_n}{\lambda_{n+1}}\right) = 1-\mu>0.$ Therefore, there exists a natural number $n_0$ for which $$\left(1-\mu\dfrac{\lambda_n}{\lambda_{n+1}}\right) = 1-\mu \ \ \forall n\geq n_0.$$
	 Therefore, we obtain from \eqref{22} that
	 \begin{eqnarray*}
	 	\dfrac{1}{1-\delta}\|u_{n+1}-\bar{v}\|^2-\dfrac{\alpha}{1-\delta}\|u_n-\bar{v}\|^2&\leq&
	 	\dfrac{1}{1-\delta}\|u_n-\bar{v}\|^2-\dfrac{\alpha}{1-\delta}\|u_{n-1}-\bar{v}\|^2+\dfrac{\alpha(1+\alpha)}{1-\delta}\|u_n-u_{n-1}\|^2\nonumber\\
	 	&&\;-\;\dfrac{1}{2}(1-\alpha)(1-\mu)\|y_{n+1}-u_n\|^2+\dfrac{1}{2}\alpha(1-\alpha)(1-\mu)\|u_n-u_{n-1}\|^2\nonumber\\
	 	&&\;-\;\dfrac{\delta(1-\alpha)}{(1-\delta)^2}\|u_{n+1}-u_n\|^2+\dfrac{\alpha\delta(1-\alpha)}{(1-\delta)^2}\|u_n-u_{n-1}\|^2 \ \ \ \ \ \ \ \forall n\geq n_0.
	 \end{eqnarray*}
	 That is,
	 \begin{eqnarray}\label{23}
	 	&&\dfrac{1}{1-\delta}\|u_{n+1}-\bar{v}\|^2-\dfrac{\alpha}{1-\delta}\|u_n-\bar{v}\|^2+ \dfrac{\delta(1-\alpha)}{(1-\delta)^2}\|u_{n+1}-u_n\|^2\nonumber\\
	 	&\leq& \dfrac{1}{1-\delta}\|u_n-\bar{v}\|^2-\dfrac{\alpha}{1-\delta}\|u_{n-1}-\bar{v}\|^2+ \dfrac{\delta(1-\alpha)}{(1-\delta)^2}\|u_n-u_{n-1}\|^2\nonumber\\
	 	&&\;+\;\left[\dfrac{\alpha(1+\alpha)}{1-\delta}+\dfrac{1}{2}\alpha(1-\alpha)(1-\mu)+\dfrac{\alpha\delta(1-\alpha)}{(1-\delta)^2}-\dfrac{\delta(1-\alpha)}{(1-\delta)^2}\right]\|u_n-u_{n-1}\|^2\nonumber\\
	 	&&\;-\;\dfrac{1}{2}(1-\alpha)(1-\mu)\|y_{n+1}-u_n\|^2 \ \ \ \ \ \ \ \ \forall n\geq n_0.
	 \end{eqnarray}
	 Let
	 \begin{eqnarray*}
	 	\varphi_n := \dfrac{1}{1-\delta}\|u_n-\bar{v}\|^2-\dfrac{\alpha}{1-\delta}\|u_{n-1}-\bar{v}\|^2+ \dfrac{\delta(1-\alpha)}{(1-\delta)^2}\|u_n-u_{n-1}\|^2. 
	 \end{eqnarray*}
	 Then, we obtain the following from \eqref{23}
	 \begin{eqnarray}\label{24}
	 	\varphi_{n+1}&\leq&\varphi_n+\left[\dfrac{\alpha(1+\alpha)}{1-\delta}+\dfrac{1}{2}\alpha(1-\alpha)(1-\mu)+\dfrac{\alpha\delta(1-\alpha)}{(1-\delta)^2}-\dfrac{\delta(1-\alpha)}{(1-\delta)^2}\right]\|u_n-u_{n-1}\|^2\nonumber\\
	 	&&\;-\;\dfrac{1}{2}(1-\alpha)(1-\mu)\|y_{n+1}-u_n\|^2 \ \ \ \ \ \ \ \ \forall n\geq n_0.
	 \end{eqnarray}
	 Now, we show that $\varphi_n\geq0\ \ \forall n\geq n_0.$\\
	 Using the fact that $\|\sum_{i=1}^{n}x_i\|^2\leq n\sum_{i=1}^{n}\|x_i\|^2,$ we have that 
	 \begin{eqnarray}\label{24b}
	 	\varphi_n &=& \dfrac{1}{1-\delta}\|u_n-\bar{v}\|^2-\dfrac{\alpha}{1-\delta}\|u_{n-1}-\bar{v}\|^2+ \dfrac{\delta(1-\alpha)}{(1-\delta)^2}\|u_n-u_{n-1}\|^2\nonumber\\
	 	&\geq& \dfrac{1}{1-\delta}\|u_n-\bar{v}\|^2-\dfrac{2\alpha}{1-\delta}\|u_n-\bar{v}\|^2-\dfrac{2\alpha}{1-\delta}\|u_n-u_{n-1}\|^2+\dfrac{\delta(1-\alpha)}{(1-\delta)^2}\|u_n-u_{n-1}\|^2\nonumber\\
	 	&=& \left(\dfrac{1}{1-\delta}-\dfrac{2\alpha}{1-\delta}\right)\|u_n-\bar{v}\|^2+\left(\dfrac{\delta(1-\alpha)}{(1-\delta)^2}-\dfrac{2\alpha}{1-\delta}\right)\|u_n-u_{n-1}\|^2\\
	 	&\geq&0,\nonumber
	 \end{eqnarray}
	 since we have from Assumption \ref{ASSUM EP} that $\alpha\leq\dfrac{1}{2}$ and $\dfrac{2\alpha}{1+\alpha}<\delta.$\\
	 
	 \noindent Next, we show that
	 $$\left[\dfrac{\alpha(1+\alpha)}{1-\delta}+\dfrac{1}{2}\alpha(1-\alpha)(1-\mu)+\dfrac{\alpha\delta(1-\alpha)}{(1-\delta)^2}-\dfrac{\delta(1-\alpha)}{(1-\delta)^2}\right]<0.$$
	 \begin{enumerate}
	 	\item [\textbf{Case I:}] If $\alpha=0$ and $0<\delta<1,$ then, it follows automatically that  
	 	$$\left[\dfrac{\alpha(1+\alpha)}{1-\delta}+\dfrac{1}{2}\alpha(1-\alpha)(1-\mu)+\dfrac{\alpha\delta(1-\alpha)}{(1-\delta)^2}-\dfrac{\delta(1-\alpha)}{(1-\delta)^2}\right]<0.$$
	 	\item [\textbf{Case II:}] If $0<\alpha\leq\dfrac{1}{2}<1$ and $0<\delta<1,$ then, notice from Assumption \ref{ASSUM EP} that
	 	$$\dfrac{(\alpha^2+2)-\sqrt{\alpha^4-8\alpha^3-8\alpha^2+4}}{2\alpha}<\delta<1.$$
	 	But for $\alpha\leq\dfrac{1}{2},$ we have that
	 	$$\dfrac{(\alpha^2+2)+\sqrt{\alpha^4-8\alpha^3-8\alpha^2+4}}{2\alpha}>1.$$
	 	This implies that
	 	\begin{eqnarray*}
	 		\dfrac{(\alpha^2+2)-\sqrt{\alpha^4-8\alpha^3-8\alpha^2+4}}{2\alpha}<\delta<\dfrac{(\alpha^2+2)+\sqrt{\alpha^4-8\alpha^3-8\alpha^2+4}}{2\alpha}.
	 	\end{eqnarray*}
	 	Thus, 
	 	\begin{eqnarray}\label{cond 1}
	 		\delta-\left(\dfrac{(\alpha^2+2)-\sqrt{(\alpha^2+2)^2-4\alpha(3\alpha+2\alpha^2)}}{2\alpha}\right)>0,
	 	\end{eqnarray}
	 	and
	 	\begin{eqnarray}\label{cond 2}
	 		\delta-\left(\dfrac{(\alpha^2+2)+\sqrt{(\alpha^2+2)^2-4\alpha(3\alpha+2\alpha^2)}}{2\alpha}\right)<0,
	 	\end{eqnarray}
	 	Multiplying \eqref{cond 1} by \eqref{cond 2}, and grouping the like terms together, we have
	 	$$\alpha\delta^2-\delta(\alpha^2+2)+3\alpha+2\alpha^2<0.$$
	 	Or equivalently,
	 	$$\alpha\delta^2-\delta\Big[2\alpha(1+\alpha)+2(1-\alpha)\Big]+2\alpha(1+\alpha)+\alpha<0.$$
	 	That is,
	 	$$2\alpha(1+\alpha)-2\alpha\delta(1+\alpha)+\alpha-2\alpha\delta+\alpha\delta^2+2\alpha\delta-2\delta(1-\alpha)<0.$$
	 	This implies that
	 	$$2\alpha(1-\delta)(1+\alpha)+\alpha(1-\delta)^2+2\alpha\delta-2\delta(1-\alpha)<0.$$
	 	Or equivalently,
	 	\begin{eqnarray*}
	 		\dfrac{\alpha(1+\alpha)}{1-\delta}+\dfrac{1}{2}\alpha+\dfrac{\alpha\delta}{(1-\delta)^2}-\dfrac{\delta(1-\alpha)}{(1-\delta)^2}<0.
	 	\end{eqnarray*}
	 	Since $\alpha, \; \delta\in(0,1),$ it follows that
	 	$$\dfrac{\alpha(1+\alpha)}{1-\delta}+\dfrac{1}{2}\alpha(1-\alpha)(1-\mu)+\dfrac{\alpha\delta(1-\alpha)}{(1-\delta)^2}-\dfrac{\delta(1-\alpha)}{(1-\delta)^2}<0.$$
	 \end{enumerate}
	 \noindent Hence, it follows from \eqref{24} that $\{\varphi_n\}_{n\in\mathbb{N}}$ is a non-increasing sequence, therefore, $\underset{n\to\infty}{\lim}\varphi_n$ exists. Thus, $\{\varphi_n\}$ is bounded. Consequently, we obtain from \eqref{24} the following
	 \begin{eqnarray}\label{25}
	 	\underset{n\to\infty}{\lim}\|y_{n+1}-u_n\|^2=0,
	 \end{eqnarray}
	 and 
	 \begin{eqnarray}\label{26}
	 	\underset{n\to\infty}{\lim}\|u_n-u_{n-1}\|^2=0.
	 \end{eqnarray}
	 From \eqref{24b}, notice that $\dfrac{1-2\alpha}{1-\delta}\|u_n-\bar{v}\|^2\leq\varphi_n.$ But $\{\varphi_n\}$ is bounded, therefore, we have that $\{\|u_n-\bar{v}\|\}$ is bounded. Hence, the sequence $\{u_n\}$ is bounded. Consequently, $\{w_n\}$ is also bounded, since $w_n=u_n+\alpha(u_n-u_{n-1}).$ Again, since $\{u_n\}, \ \{w_n\}$ are bounded, it ensues that $\{y_n\}$ is also bounded.
  \end{proof}

\begin{theorem}\label{THM2}
	Let Assumptions \ref{B} and \ref{ASSUM EP} be satisfied. Then the sequences, $\{y_n\}$ and $\{w_n\},$ produced by Algorithm \ref{ALG} converge weakly to a solution of EP\eqref{1.1}.
\end{theorem}
\begin{proof}\noindent\\
	Using $w_n = u_n+\alpha(u_n-u_{n-1}),$ we obtain 
	$$\underset{n\to\infty}{\lim}\|w_n-u_n\|=\alpha\underset{n\to\infty}{\lim}\|u_n-u_{n-1}\| =0.$$
	Also, $$\|y_{n+1}-w_n\|\leq\|y_{n+1}-u_n\|+\|w_n-u_n\|\longrightarrow0.$$
	Again, using the definition of $\{u_n\},$ we have
	$$\|y_n-u_n\|=\delta\|y_n-w_{n-1}\|\longrightarrow0.$$
	Therefore,
	$$\|y_n-w_n\|\leq\|u_n-y_n\|+\|w_n-u_n\|\longrightarrow0.$$
	From \eqref{15}, we get
	\begin{eqnarray*}
		\left(1-\mu\dfrac{\lambda_n}{\lambda_{n+1}}\right)\Big(\|w_n-z_n\|^2+\|z_n-y_{n+1}\|^2\Big)&\leq&\|w_n-\bar{v}\|^2-\|y_{n+1}-\bar{v}\|^2\\
		&=&\Big(\|w_n-\bar{v}\|+\|y_{n+1}-\bar{v}\|\Big)\Big(\|w_n-\bar{v}\|-\|y_{n+1}-\bar{v}\|\Big)\\
		&\leq& \Big(\|w_n-\bar{v}\|+\|y_{n+1}-\bar{v}\|\Big)\|y_{n+1}-w_n\|^2\\
		&\leq& \varPhi_1\|y_{n+1}-w_n\|\longrightarrow0,
	\end{eqnarray*}
	where $\varPhi_1 = \underset{n\leq1}{\sup}\{\|w_n-\bar{v}\|+\|y_{n+1}-\bar{v}\|\}.$ Since $(1-\mu)>0,$ we have
	$$\underset{n\to\infty}{\lim}\|w_n-z_n\|=\underset{n\to\infty}{\lim}\|z_n-y_{n+1}\|=0.$$
	Because $\{y_n\}$ is bounded, there is a subsequence $\{y_{n_j}\}$ of $\{y_n\}$ that converges weakly to $r^*$ as $j\to\infty.$ So, $\|y_n-w_n\|\longrightarrow0$ implies that $w_{n_j}\rightharpoonup r^*.$ Also, $\|y_{n+1}-z_n\|\longrightarrow0$ implies that $z_{n_j}\rightharpoonup r^*.$ More so, $u_{n_j}\rightharpoonup r^*.$  But $z_n\in\mathcal{C} \ \forall n\in\mathbb{N},$ and since $\mathcal{C}$ is a closed convex set, it is weakly closed. Hence, $r^*\in\mathcal{C}.$\\
	
	\noindent Now, we show that $r^*\in EP(f,\mathcal{C}).$  
	
	\noindent From \eqref{13}, we have
		\begin{eqnarray}\label{27}
		\lambda_n(f(w_n, y_{n+1})-f(w_n, z_n)) -\frac{\mu\lambda_n\Bigl(\|w_n-z_n\|^2+\|y_{n+1}-z_n\|^2\Bigr)}{2\lambda_{n+1}}\leq \lambda_nf(z_n, y_{n+1}).
	\end{eqnarray}
	We obtain from \eqref{6} that
	\begin{eqnarray}\label{28}
		\lambda_nf(z_n,y_{n+1})\leq \lambda_nf(z_n,z)-\langle w_n-y_{n+1}, z-y_{n+1}\rangle \ \ \forall z\in\mathcal{C}.
	\end{eqnarray}
	So, when we combine \eqref{27} with \eqref{28}, we have
	
	\begin{eqnarray}\label{29}
		&&\lambda_n(f(w_n, y_{n+1})-f(w_n, z_n)) -\frac{\mu\lambda_n\Bigl(\|w_n-z_n\|^2+\|y_{n+1}-z_n\|^2\Bigr)}{2\lambda_{n+1}}\nonumber\\
		&\leq& \lambda_nf(z_n,z)-\langle w_n-y_{n+1}, z-y_{n+1}\rangle \ \ \ \ \forall z\in\mathcal{C}.
	\end{eqnarray}
	Additionally, from \eqref{5}, we obtain
	\begin{eqnarray}\label{30}
		\lambda_n(f(w_n,z_n)-f(w_n,y_{n+1}))\leq -\langle w_n-z_n, y_{n+1}-z_n\rangle.
	\end{eqnarray}
	Adding \eqref{29} and \eqref{30} together give
	\begin{eqnarray*}\label{31}
	\lambda_nf(z_n,z)&\geq& \langle w_n-z_n, y_{n+1}-z_n\rangle-\frac{\mu\lambda_n\Bigl(\|w_n-z_n\|^2+\|y_{n+1}-z_n\|^2\Bigr)}{2\lambda_{n+1}}\nonumber\\
	&&\;+\;\langle w_n-y_{n+1}, z-y_{n+1}\rangle \ \ \ \ \forall z\in\mathcal{C}.	
	\end{eqnarray*}
	Recall that $\underset{n\to\infty}{\lim}\lambda_n$ exists. That is, $\underset{n\to\infty}{\lim}\lambda_n=\lambda>0.$ So from the last inequality, we obtain
	\begin{eqnarray*}
		f(z_{n_j},z)&\geq& \dfrac{1}{\lambda_{n_j}}\langle w_{n_j}-z_{n_j}, y_{n_{j+1}}-z_{n_j}\rangle-\frac{\mu\Bigl(\|w_{n_j}-z_{n_j}\|^2+\|y_{n_{j+1}}-z_{n_j}\|^2\Bigr)}{2\lambda_{n_{j+1}}}\nonumber\\
		&&\;+\;\dfrac{1}{\lambda_{n_j}}\langle w_{n_j}-y_{n_{j+1}}, z-y_{n_{j+1}}\rangle \ \ \ \ \forall z\in\mathcal{C}.
	\end{eqnarray*}
	Letting $j\to\infty$ in the last inequality, and using  conditions (B1) and (B2) of Assumption \ref{B} and the facts that $\{w_{n_j}\}, \ \{y_{n_j}\}, \ \{z_{n_j}\}$ converge weakly to $r^*,$ we obtain that $f(r^*, z)\geq0 \ \forall z\in\mathcal{C}.$ Hence, $r^*\in EP(f,\mathcal{C}).$\\
	
	\noindent Next, we show that $\underset{n\to\infty}{\lim}\{\|y_n-r^*\|\}$ exists.\\
	From the definition of $\varphi_n,$ we have 
	\begin{eqnarray}\label{32}
		\dfrac{1}{1-\delta}\|u_n-r^*\|^2 = \varphi_n+\dfrac{\alpha}{1-\delta}\|u_{n-1}-r^*\|^2+\dfrac{\delta(1-\alpha)}{(1-\delta)^2}\|u_n-u_{n-1}\|^2.
	\end{eqnarray}
	Notice that 
	$$\|u_{n-1}-r^*\|^2 = \|u_{n-1}-u_n\|^2+2\langle u_{n-1}-u_n, u_n-r^*\rangle+\|u_n-r^*\|^2.$$
	Using the last inequality in \eqref{32}, we get
	\begin{eqnarray*}
		\dfrac{1-\alpha}{1-\delta}\|u_n-r^*\|^2 = \varphi_n+\dfrac{2\alpha}{1-\delta}\langle u_{n-1}-u_n, u_n-r^*\rangle+\dfrac{\alpha+\delta(1-2\alpha)}{(1-\delta)^2}\|u_n-u_{n-1}\|^2.
	\end{eqnarray*}
	But $\underset{n\to\infty}{\lim}\varphi_n$ exists, $\{u_n\}$ is bounded and $\underset{n\to\infty}{\lim}\|u_n-u_{n-1}\| =0.$ Therefore, $\underset{n\to\infty}{\lim}\|u_n-r^*\|$ exists for any\\
	 $r^*\in EP(f,\mathcal{C}).$ Hence, we obtain from
	 $$\|y_n-r^*\|^2 = \|y_n-u_n\|^2+2\langle y_n-u_n, u_n-r^*\rangle+\|u_n-r^*\|^2$$
	 that $\underset{n\to\infty}{\lim}\|y_n-r^*\|$ exists for any $r^*\in EP(f,\mathcal{C}).$ Thus, we conclude from Lemma \ref{D1} that $\{y_n\}$ converges weakly to a point in $EP(f,\mathcal{C}).$
\end{proof}

 \section{Convergence Rate}\label{Convergence Rate}
\noindent This section presents the linear convergence rate theorem. In this theorem, we establish that the sequence generated by Algorithm \ref{ALG} converges uniquely to the set of solutions of EP\eqref{1.1}, assuming the $\beta-$strong pseudomonotonicity of the bifunction $f.$

\hfill

\begin{theorem}
	Let $f:\mathcal{C}\times\mathcal{C}\longrightarrow\mathbb{R}$ be a $\beta-$strongly pseudomonotone bifunction satisfying $(B2),$ $(B3)$ and $(B5)$ of Assumption \ref{B}. Then the sequence produced by Algorithm \ref{ALG} converges to the unique solution  $r^*$ of EP\eqref{1.1} with an $R-$linear rate when $\alpha=0$ and $0<\delta<1$.
\end{theorem}

\begin{proof}\noindent\\
From \eqref{7}, we have
\begin{eqnarray}\label{33}
	-\lambda_nf(z_n,y_{n+1})\geq \langle w_n-y_{n+1}, r^*-y_{n+1}\rangle -\lambda_nf(z_n,r^*) \ \ \forall r^*\in EP(f,\mathcal{C}).
\end{eqnarray}
Since $r^*\in EP(f,\mathcal{C})$ and $z_n\in\mathcal{C},$ it follows that $f(r^*, z_n)\geq0.$ Thus, by the strong pseudomonotonicity assumption on $f$, we have that
$$f(z_n,r^*)\leq-\beta\|z_n-r^*\|^2.$$
This implies from \eqref{33} that
\begin{eqnarray}\label{34}
	-\lambda_nf(z_n,y_{n+1})\geq \langle w_n-y_{n+1}, r^*-y_{n+1}\rangle +\lambda_n\beta\|z_n-r^*\|^2 \ \ \forall r^*\in EP(f,\mathcal{C}).
\end{eqnarray}
Combining \eqref{5} with \eqref{34} gives us that
\begin{eqnarray}\label{35}
	2\lambda_n\Big(f(w_n, y_{n+1})-f(w_n,z_n)-f(z_n,y_{n+1})\Big)&\geq&2\langle w_n-z_n,y_{n+1}-z_n\rangle+2\langle w_n-y_{n+1}, r^*-y_{n+1}\rangle\nonumber\\
	&&\;+\;\lambda_n\beta\|z_n-r^*\|^2\nonumber\\
	&=& \|w_n-z_n\|^2+\|y_{n+1}-z_n\|^2-\|y_{n+1}-w_n\|^2+\|y_{n+1}-w_n\|^2\nonumber\\
	&&\;+\; \|y_{n+1}-r^*\|^2-\|w_n-r^*\|^2+2\lambda_n\beta\|z_n-r^*\|^2\nonumber\\
	&=&\|w_n-z_n\|^2+\|y_{n+1}-z_n\|^2+\|y_{n+1}-r^*\|^2\nonumber\\
	&&\;-\;\|w_n-r^*\|^2+2\lambda_n\beta\|z_n-r^*\|^2.
\end{eqnarray}
By the definition of $\lambda_n,$ we obtain
\begin{eqnarray}\label{36}
	2\lambda_n\Big(f(w_n, y_{n+1})-f(w_n, z_n)-f(z_n, y_{n+1})\Big)\leq\dfrac{\lambda_n \mu}{\lambda_{n+1}}\Big(\|w_n-z_n\|^2+\|z_n-y_{n+1}\|^2\Big).
\end{eqnarray}
Using \eqref{36} in \eqref{35} gives
\begin{eqnarray}\label{37}
	\|y_{n+1}-r^*\|^2&\leq&\|w_n-r^*\|^2-\left(1-\mu\dfrac{\lambda_n}{\lambda_{n+1}}\right)\|w_n-z_n\|^2\nonumber\\
	&&\;-\;\left(1-\mu\dfrac{\lambda_n}{\lambda_{n+1}}\right)\|y_{n+1}-z_n\|^2
	-2\lambda_n\beta\|z_n-r^*\|^2.
\end{eqnarray}
	Let
$$\gamma:=\min\left\{1-\mu, \lambda\beta \right\},$$
where $\lambda:=\underset{n\to\infty}{\lim}\lambda_n,$ we obtain
$$\underset{n\to\infty}{\lim}\left(1-\mu\dfrac{\lambda_n}{\lambda_{n+1}}\right)= 1-\mu\geq\gamma,$$
$$\underset{n\to\infty}{\lim}\lambda_n\beta = \lambda\beta \geq \gamma.$$
Therefore, there exists $N_1\in\mathbb{N}$ such that 
$$\left(1-\mu\dfrac{\lambda_n}{\lambda_{n+1}}\right)\geq\gamma \ \ \ \ \forall n\geq N_1 \ \mbox{and}$$
$$\lambda_n\beta\geq\gamma \ \ \ \ \ \forall n\geq N_1.$$
So, using \eqref{37} it follows that
\begin{eqnarray*}
	\|y_{n+1}-r^*\|^2&\leq&\|w_n-r^*\|^2-\gamma\|w_n-z_n\|^2-\gamma\|y_{n+1}-z_n\|^2-2\gamma\|z_n-r^*\|^2\\
	&\leq& \|w_n-r^*\|^2-\dfrac{\gamma}{2}\|w_n-r^*\|^2-\dfrac{\gamma}{2}\|y_{n+1}-r^*\|^2.
\end{eqnarray*}
This implies
\begin{eqnarray}\label{38}
	\|y_{n+1}-r^*\|^2&\leq&\left(\dfrac{1-\dfrac{\gamma}{2}}{1+\dfrac{\gamma}{2}}\right)\|w_n-r^*\|^2\nonumber\\
	&=&\omega\|w_n-r^*\|^2,
\end{eqnarray}
where $\omega= \left(\dfrac{1-\dfrac{\gamma}{2}}{1+\dfrac{\gamma}{2}}\right).$\\

\noindent Recall that $y_{n+1}=\dfrac{1}{1-\delta}u_{n+1}-\dfrac{\delta}{1-\delta}w_n.$\\
Therefore, 
\begin{eqnarray}\label{39}
	\|y_{n+1}-r^*\|^2&=&\dfrac{1}{1-\delta}\|u_{n+1}-r^*\|^2-\dfrac{\delta}{1-\delta}\|w_n-r^*\|^2+\dfrac{\delta}{(1-\delta)^2}\|u_{n+1}-w_n\|^2\nonumber\\
	&=& \dfrac{1}{1-\delta}\|u_{n+1}-r^*\|^2-\dfrac{\delta}{1-\delta}\|w_n-r^*\|^2+\delta\|y_{n+1}-w_n\|^2.
\end{eqnarray}
Also, recall that $w_n=u_n+\alpha(u_n-u_{n-1}).$ This implies that
$$\|w_{n+1}-r^*\|^2 = \|u_{n+1}-r^*\|^2,$$
since $\alpha=0.$ Therefore, \eqref{39} gives us that
\begin{eqnarray}\label{40}
	\|y_{n+1}-r^*\|^2 = \dfrac{1}{1-\delta}\|w_{n+1}-r^*\|^2-\dfrac{\delta}{1-\delta}\|w_n-r^*\|^2+\delta\|y_{n+1}-w_n\|^2.
\end{eqnarray}
Using \eqref{40} in \eqref{38}, we have
\begin{eqnarray*}
	\dfrac{1}{1-\delta}\|w_{n+1}-r^*\|^2&\leq&\left(\omega+\dfrac{\delta}{1-\delta}\right)\|w_n-r^*\|^2-\delta\|y_{n+1}-w_n\|^2\\
	&\leq&\dfrac{1}{1-\delta}\Big[\omega(1-\delta)+\delta\Big]\|w_n-r^*\|^2.
\end{eqnarray*}
That is,
\begin{eqnarray*}
	\|w_{n+1}-r^*\|^2	&\leq&\rho\|w_n-r^*\|^2
\end{eqnarray*}
where $\rho = \Big[\omega(1-\delta)+\delta\Big].$\\
This implies
$$\|w_{n+1}-r^*\|\leq\rho\|w_n-r^*\|^2\leq\cdots\leq\rho^{n-N_1+1}\|w_{N_1}-r^*\|^2=\dfrac{\rho}{\rho^{N_1}}\|w_{N_1}-r^*\|^2\rho^n.$$
Thus, $\{w_n\}$ converges $R-$linearly to $r^*$. Consequently, $\{y_n\}$ converges $R-$linearly to $r^*.$
\end{proof}

\section{Numerical Examples}\label{numerics}
\noindent In this section, we compare the numerical performance of our proposed Algorithm \ref{ALG} with some algorithms in the literature using three (3)  examples. We consider the following algorithms: \cite[Algorithm 1]{VIN}, \cite[Algorithm 4.1]{Hoai},  \cite[Algorithm 3.1]{YangOPTL}, \cite[Algorithm 3.1]{Yang} and  \cite[Algorithm 3.1]{YinLiuYang}. 
\vskip 2mm

\begin{exm}\label{ex1}
Consider the following nonmonotone bifunction arising in  Nash-Cournot equilibrium model in $\mathbb{R}^m$. The bifunction $f$ is defined by  
\begin{eqnarray}\label{vcx}
f(x,y):= \langle Px+Qy+r, y-x \rangle,
\end{eqnarray}
\noindent
where $P =(p_{ij})_{m\times m}$ and $Q =(q_{ij})_{m\times m}$ are $m\times m$ symmetric positive semidefinite matrices with $P-Q$ positive semidefinite and $r\in \mathbb{R}^m$. The bifunction $f$ has the form of the one arising from a Nash-Cournot oligopolistic electricity market equilibrium model \cite{CKK2004} and that $f$ is convex in $y$, Lipschitz-type continuous with constants $\kappa_1 = \kappa_2 = \frac{1}{2}\|P-Q\|_2$, and the positive semidefinition of $P-Q$ implies that $f$ is monotone (hence pseudomonotone).  We choose $P$ and $Q$ to be of the form $A^TA$ with $A =(a_{ij})_{m\times m}$ randomly generated in $\mathbb{R}$, and starting point randomly generated in $\mathbb{R}$ with $C: = \displaystyle\prod_{i=1}^m [-10,10]$ with $m=50, 100, 200$ and $300$.

\end{exm}
\begin{itemize}
\item  In our proposed Algorithm \ref{ALG}, we take $\alpha =0.1, \delta=0.9, \mu= 10^{-5}$ and $\lambda_0=0.1$. 
\item In \cite[Algorithm 1]{VIN} (VM Alg. 1 for short) we take $\theta=0.6,~\epsilon_n=\frac{1}{(n+1)^4}$ and $\lambda=10^{-3}$.
\item In \cite[Algorithm 4.1]{Hoai}, (PNN Alg. 4.1 for short) we take $\varphi=\frac{\sqrt{5}+1}{2}$ and $\lambda_n=\frac{1}{n+1}$. 
\item  In \cite[Algorithm 3.1]{Yang} (YL Alg. 3.1 for short), we take $\mu=0.98$, $\delta=0.1, ~\alpha=0.98,~ \theta=0.75$, and $\lambda_1=0.1$. 
\item In  \cite[Algorithm 3.1]{YangOPTL},(YL36 Alg. 3.1 for short) we take $\mu=0.9$, $\lambda_0=0.001$, $\beta_n=0.5$, $\alpha_n=\frac{1}{1000(n+1)}$. 
\item In \cite[Algorithm 3.1]{YinLiuYang} (YLY Alg. 3.1 for short) we take $\mu=0.01$, $\theta=0.8, ~\lambda_0=\lambda_1=0.001$.   
\end{itemize}
 The proximity operator in all these algorithms is computed using the MATLAB built-in function $fmincon$ using $\mathcal{C}$ defined above. Furthermore, for $m=50, 100, 200$ and $300$, the initial points are generated using the formula:
 \[
  y_{-1}=\Big(\frac{1}{11}, \frac{2}{21}, \frac{3}{31}, \cdots , \frac{m}{10m+1}\Big)^T \mbox{ and }  w_{-2}= \Big( \frac{6}{2}, \frac{7}{5}, \frac{8}{10}, \cdots , \frac{m+5}{m^2+1} \Big)^T. \mbox{ Then set } ~y_0=y_{-1} \mbox{ and } w_{-1}=w_{-2}.
 \]
 
 \noindent  The iteration process is terminated when $n=5001$ or $E_n=\|x_{n+1}-x_n\|>10^{-6}$. The outputs of the experiments are presented in Table \ref{T1} and Figures  \ref{F1} and \ref{F2}.

\begin{table}[H]
\centering
\caption{Numerical performance of all algorithms in Example \ref{ex1}}
\begin{tabular}{lcccccccc}
\toprule
\multicolumn{1}{c}{\multirow{2}[4]{*}{Algorithms}} & \multicolumn{2}{c}{$m=50$} & \multicolumn{2}{c}{$m=100$} & \multicolumn{2}{c}{$m=200$} & \multicolumn{2}{c}{$m=300$} \\
\cmidrule{2-9}          & Iter.    & Time (s)   & Iter.     & Time (s)   & Iter.     & Time (s)   & Iter.     & Time (s) \\
\midrule
Algorithm \ref{ALG} & 8     & 0.0160    & 9    & 0.0262   & 9   & 0.0356   &  9    & 0.0779   \\  
VM Alg. 1           & 525   & 17.3327   & 292  & 15.4211  & 196 & 30.9796  & 136   & 50.7894  \\
PNN Alg. 4.1        & 431   & 30.1965   & 740  & 64.6567  & 799 & 153.2892 & 1009  & 277.7423 \\
YL Alg. 3.1         & 251   & 7.2245    & 335  & 16.5650  & 278 & 34.6320  & 338   & 66.3562  \\
YL36 Alg. 3.1       & 207   & 3.4590    & 205  & 5.5486   & 205 & 10.9219  & 387   & 40.7502  \\ 
YLY Alg. 3.1        & 167   & 4.7193    & 135  & 5.5824   & 357 & 37.4116  & 390   & 60.3725 \\ 
\bottomrule
\end{tabular} \label{T1}
\end{table}

\begin{figure}[H]
\begin{center}
\includegraphics[width=8.0cm,height=10cm]{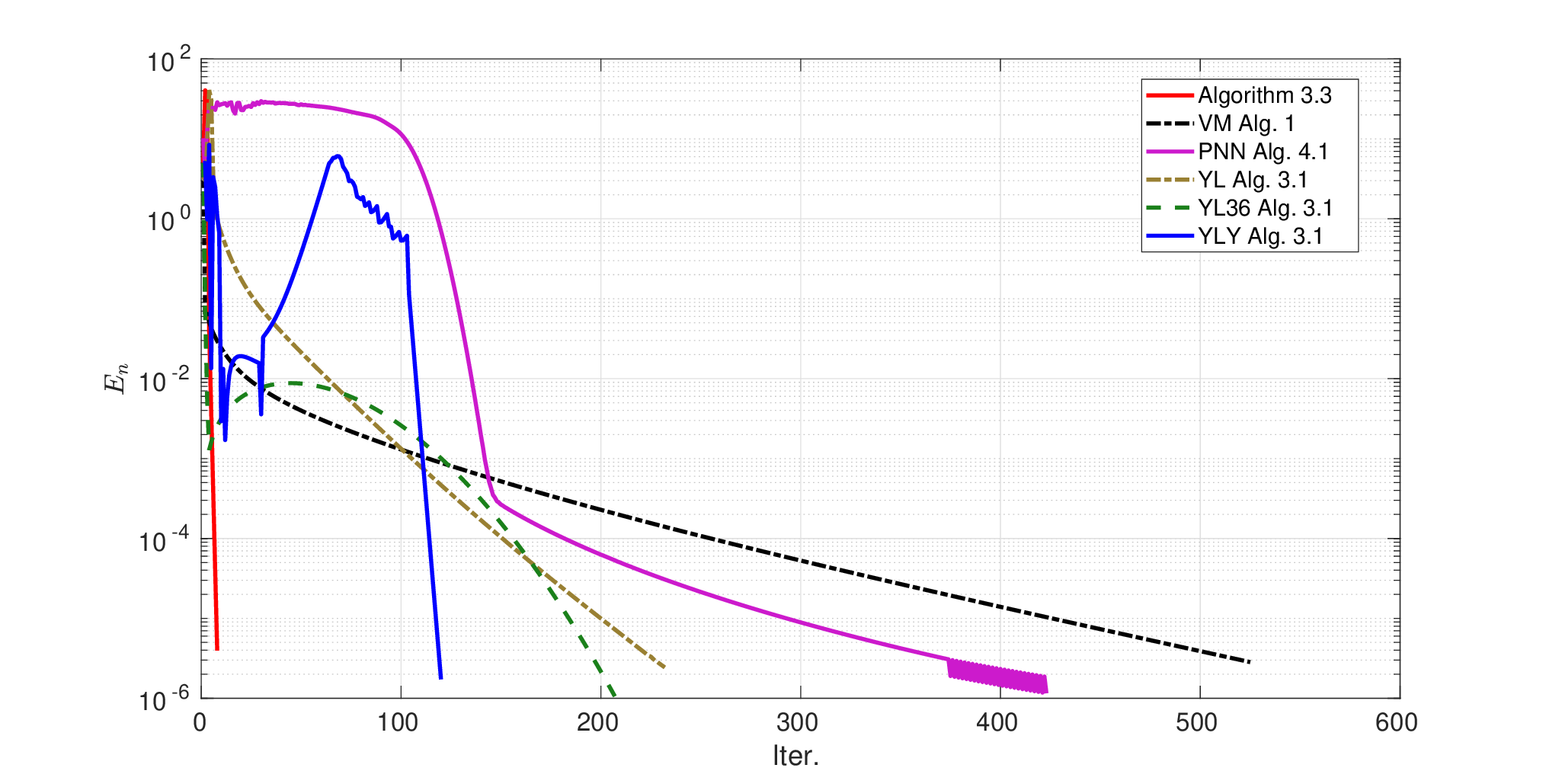} %
\includegraphics[width=8.0cm,height=10cm]{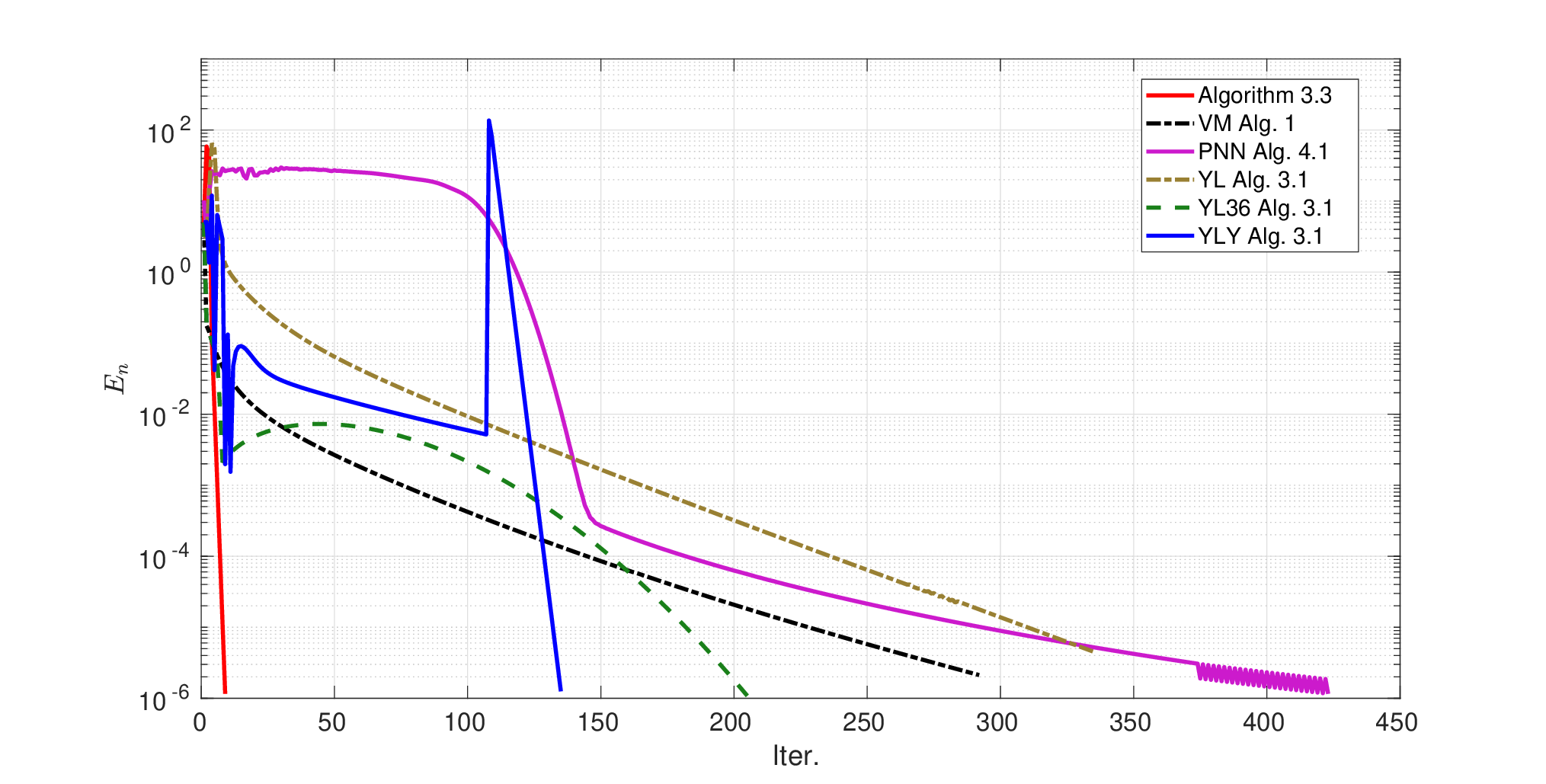} %
\end{center}
\caption{Numerical results for Example \ref{ex1} for Cases I and II, respectively} \label{F1}
\end{figure}

\begin{figure}[H]
\begin{center} 
\includegraphics[width=8.0cm,height=10cm]{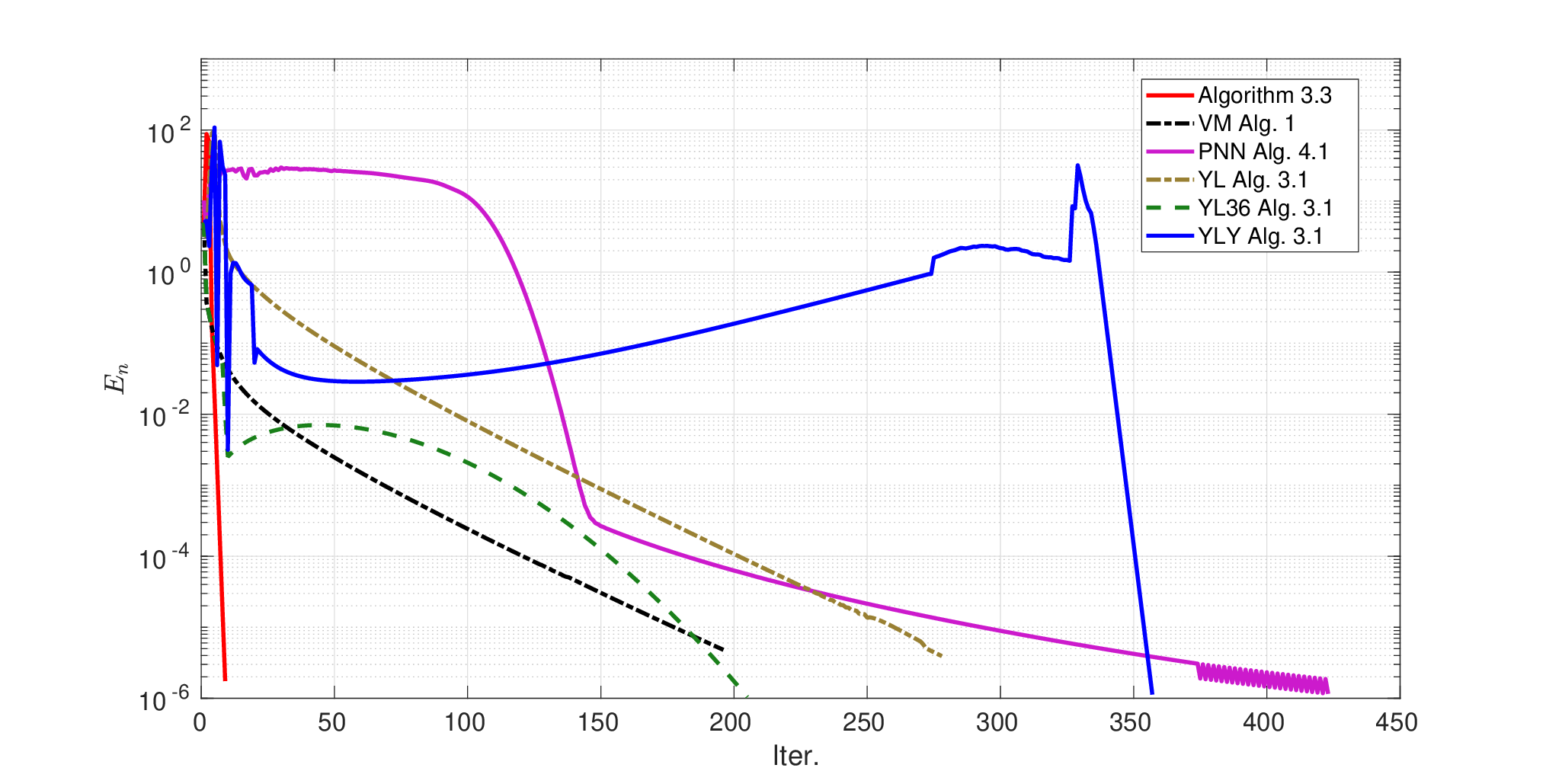}
\includegraphics[width=8.0cm,height=10cm]{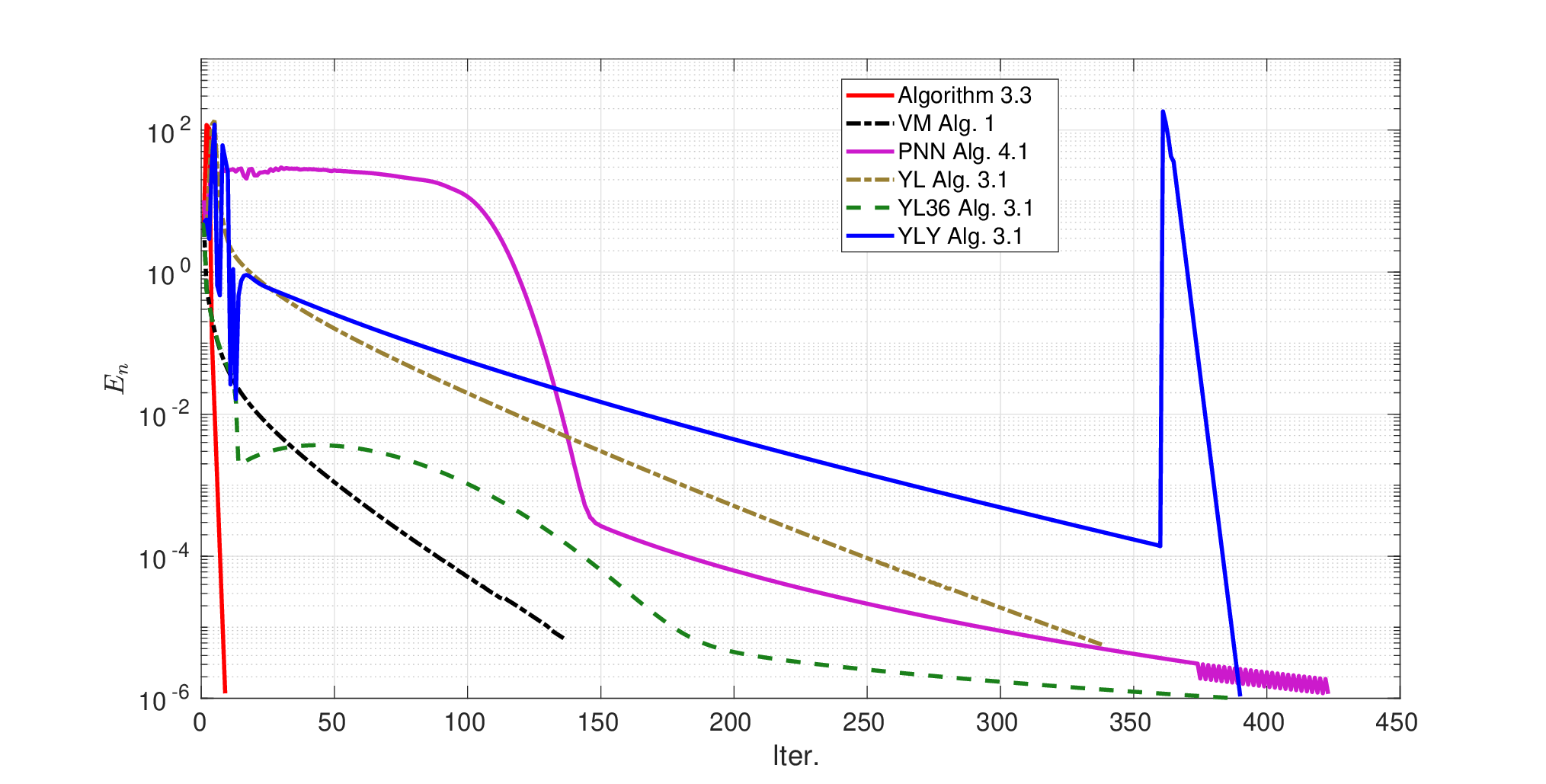}
\end{center}
\caption{Numerical results for Example \ref{ex1} for Cases III and IV, respectively} \label{F2}
\end{figure}

\hfill

\begin{exm}\label{ex2}
Let $A$ be a square matrix  defined by 
$A:=(a_{ij})_{1\leq i,j\leq m}$, with entries $a_{ij}$  given by

\begin{eqnarray*}
a_{ij}=
\left\{  \begin{array}{llll}
         & -1,~~{\rm if} ~~j=m+1-i~~{\rm and}~~j>i,\\
         & 1,~~{\rm if} ~~j=m+1-i~~{\rm and}~~j<i,\\
         & 0~~{\rm otherwise}.
         \end{array}
         \right.
\end{eqnarray*}
Let $\mathcal{C}$ be the $m$ dimensional cube $[-1,1]^m$. Define the bifunction $f$ by  $f(x,y)=\langle Ax, y-x \rangle$. Using this bifunction, we will evaluate the performance of each algorithm by choosing the following  set of control parameters:  
\end{exm}

\begin{itemize}
\item  In our proposed Algorithm \ref{ALG}, we take the same set of parameters used in Example \ref{ex1} and to get more iterates, we choose $ \delta=0.5$. 
\item In \cite[Algorithm 1]{VIN} (VM Alg. 1 for short) we take $\theta=0.6,~\epsilon_n=\frac{1}{(n+1)^4}$ and $\lambda=0.5$.
\item In \cite[Algorithm 4.1]{Hoai}, (PNN Alg. 4.1 for short) we take $\varphi=\frac{\sqrt{5}+1}{2}$ and $\lambda_n=10^{-8}$. 
\item  In \cite[Algorithm 3.1]{Yang} (YL Alg. 3.1 for short), we take $\mu=10^{-8}$, $\delta=\frac{6}{7}, ~\alpha=0.9,~ \theta=0.25$, and $\lambda_1=10^{-5}$. 
\item In  \cite[Algorithm 3.1]{YangOPTL},(YL36 Alg. 3.1 for short) we take $\mu=10^{-2}$, $\lambda_0=10^{-8}$, $\beta_n=0.5$, $\alpha_n=\frac{1}{1000(n+1)}$. 
\item In \cite[Algorithm 3.1]{YinLiuYang} (YLY Alg. 3.1 for short) we take $\mu=10^{-4}$, $\theta=0.8, ~\lambda_0=\lambda_1=0.001$.   
\end{itemize}
The respective algorithms generate the initial points randomly in $\mathbb{R}^m$, $m=500,~1000,~2000$ and $3000$. The iteration process ends when $n=5001$ or $E_n=\|x_{n+1}-x_n\|>10^{-6}$. The results of the experiments are presented in Table \ref{T2} and Figures \ref{F3}.

\begin{table}[H]
\centering
\caption{Numerical performance of all algorithms in Example \ref{ex2}}
\begin{tabular}{lcccccccc}
\toprule
\multicolumn{1}{c}{\multirow{2}[4]{*}{Algorithms}} & \multicolumn{2}{c}{$m=500$} & \multicolumn{2}{c}{$m=1000$} & \multicolumn{2}{c}{$m=2000$} & \multicolumn{2}{c}{$m=3000$} \\
\cmidrule{2-9}          & Iter.    & Time (s)   & Iter.     & Time (s)   & Iter.     & Time (s)   & Iter.     & Time (s) \\
\midrule
Algorithm \ref{ALG} & 7     & 0.0067    & 8    & 0.0571  & 8   & 0.1874  &  9    & 0.4415   \\  
VM Alg. 1           & 158   & 0.1164    & 165  & 0.5104  & 174 & 1.4672  & 181   & 3.0611  \\
PNN Alg. 4.1        & 32    & 0.0179    & 33   & 0.0576  & 34  & 0.1702  & 34    & 0.3248 \\
YL Alg. 3.1         & 91    & 0.0722    & 94   & 0.5047  & 96  & 1.4862  & 97    & 3.1344  \\
YL36 Alg. 3.1       & 236   & 0.2921    & 238  & 1.6026  & 241 & 5.3357  & 243   & 11.4872  \\ 
YLY Alg. 3.1        & 1649  & 1.5758    & 1734 & 7.9871  & 1560& 23.6963 & 1341  & 42.8925 \\ 
\bottomrule
\end{tabular} \label{T2}
\end{table}

\begin{figure}[H]
\begin{center}
\includegraphics[width=8.0cm,height=10cm]{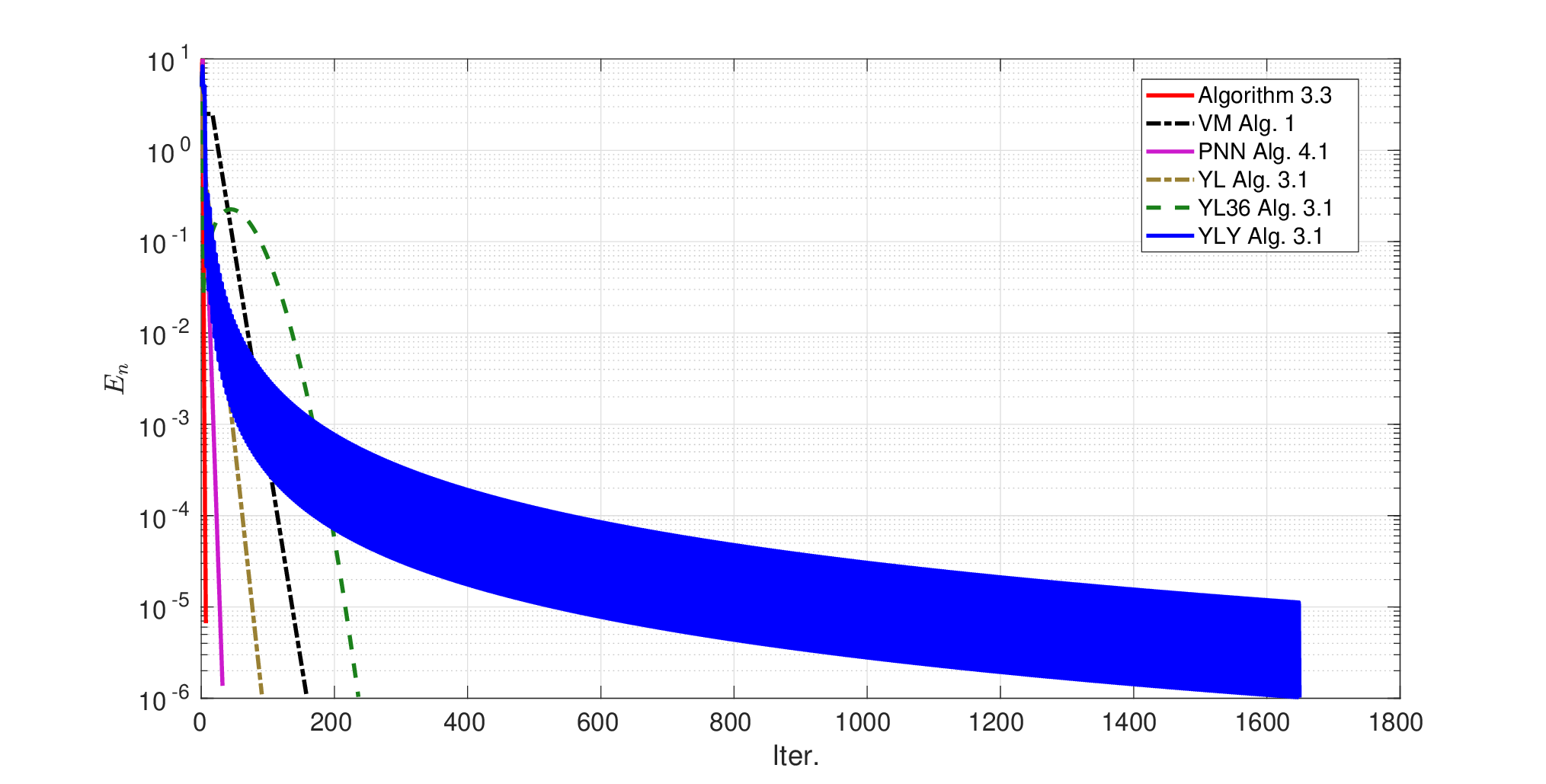} %
\includegraphics[width=8.0cm,height=10cm]{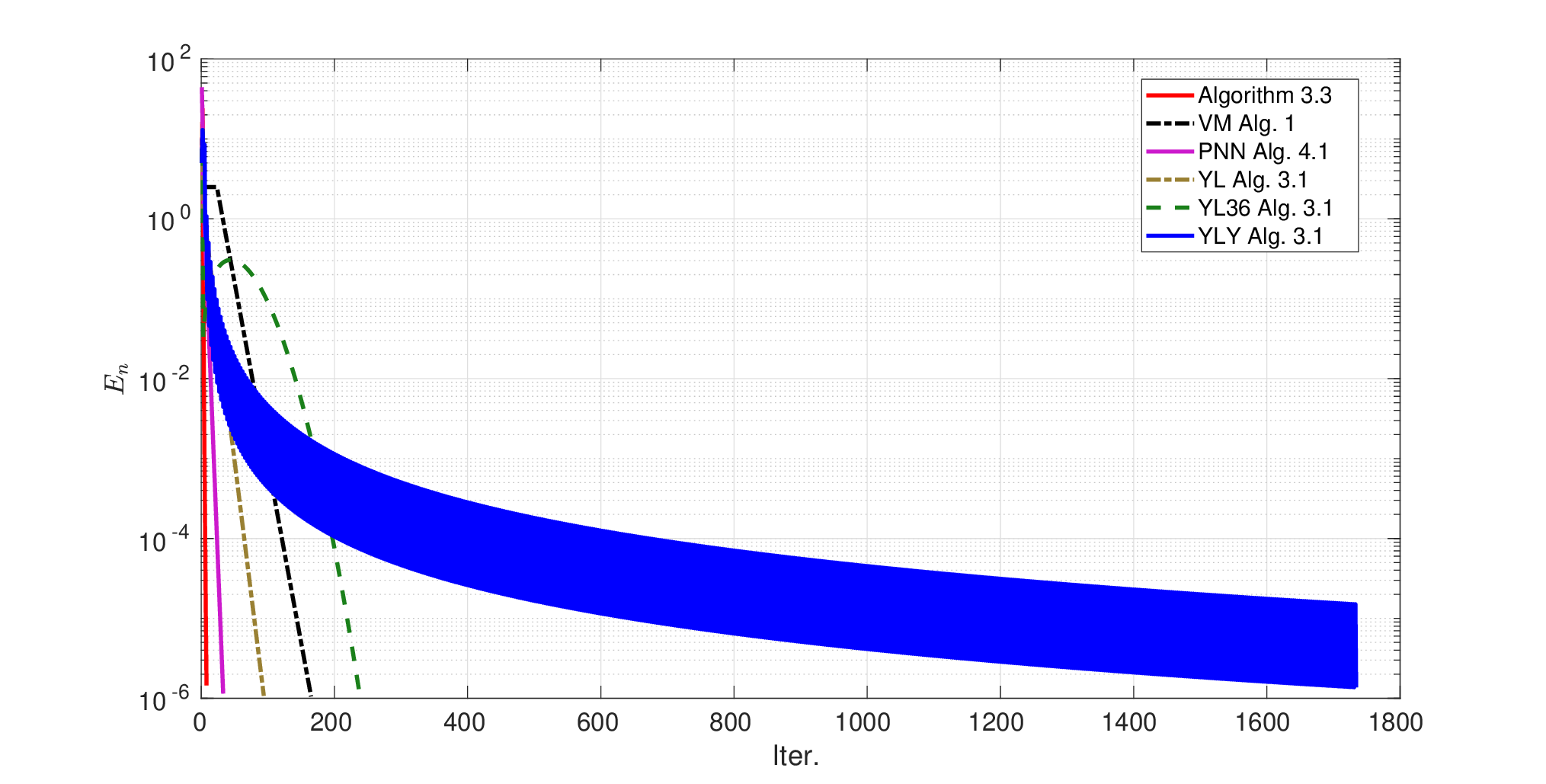} %
\includegraphics[width=8.0cm,height=10cm]{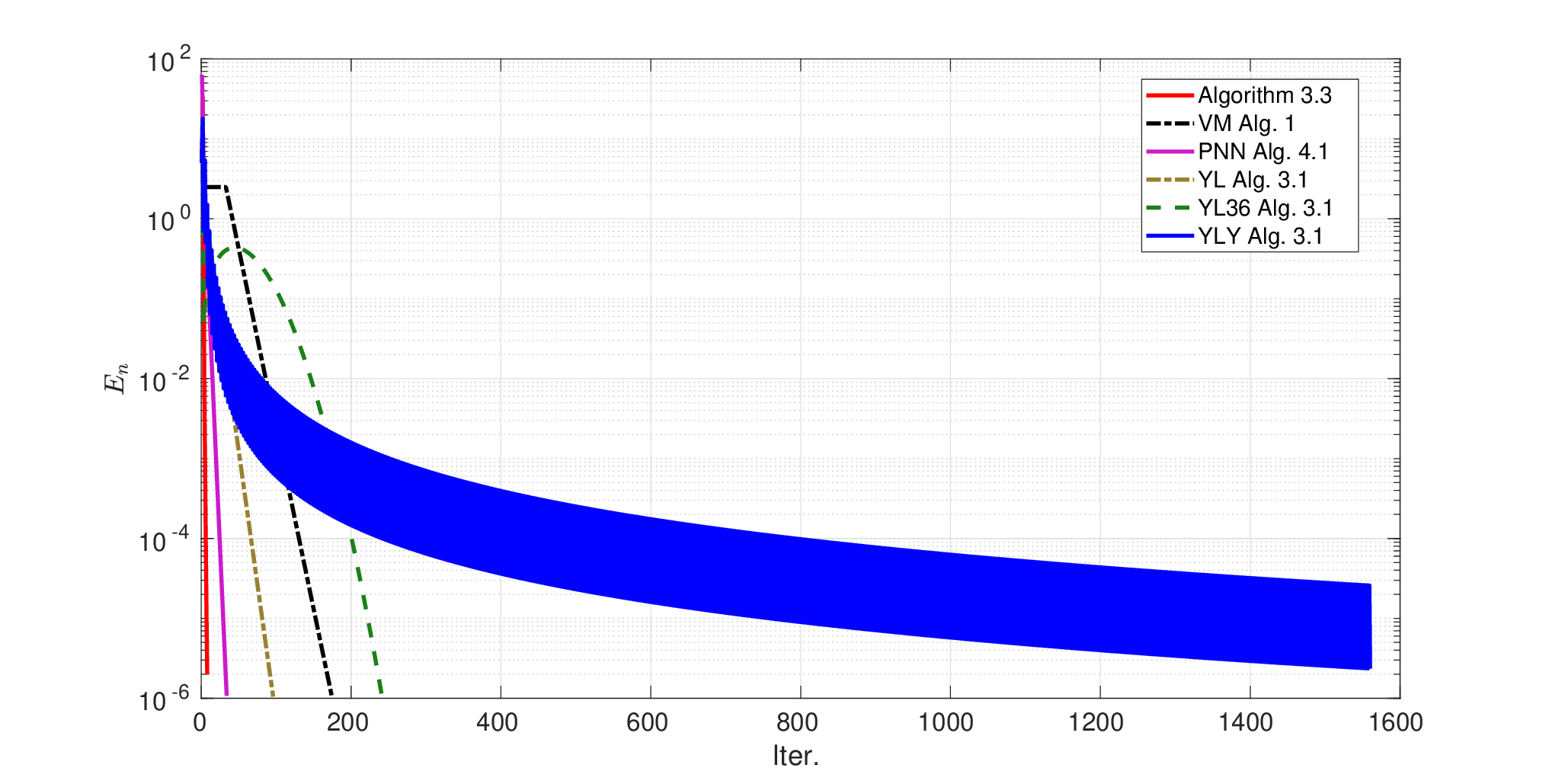} %
\includegraphics[width=8.0cm,height=10cm]{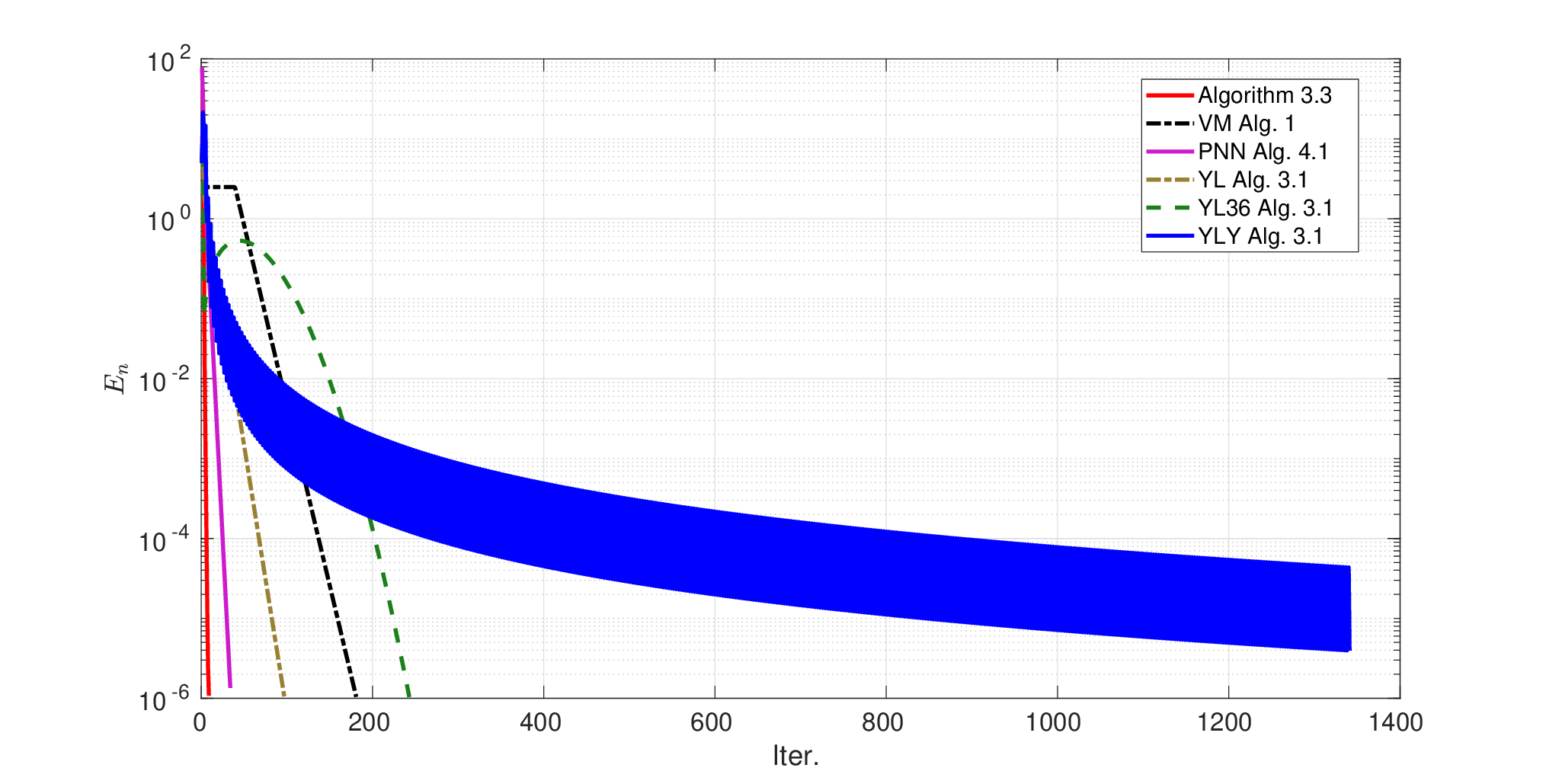} %
\end{center}
\caption{Numerical results for Example \ref{ex2} for Cases I-IV, respectively} \label{F3} 
\end{figure}
	
	 \hfill

\begin{exm}\label{EX3}
Let 
 $\mathbb{H}:=L^2([0,1])$ and $\mathcal{C}:=\{x \in \mathbb{H}: \|x\|_2 \leq 2\}$. Define
$$
Ax(t):=e^{-\|x\|_2} \int_{0}^{t} x(s) ds,~~\forall x \in L^2([0,1]), t \in [0,1].
$$ 
\noindent
Then, consider the bifunction $f$ defined by $f(x,y):=\langle A(x),y-x\rangle,~~\forall x,y \in \mathcal{H}$. Using this bifunction, we will evaluate the performance of these algorithms under consideration in the setting of the classical Hilbert space.  For the purpose of numerical efficiency, the interval $[0,1]$ is descritized into 500 points using the matlab syntax $linspace(0, 1, n)'$. The following set of parameters are used for implementation of each algorithm:

\begin{itemize}
\item 
In our proposed Algorithm \ref{ALG}, we take $\alpha =0.1, \delta=0.95, \mu= 10^{-5}$ and $\lambda_0=0.001$. 
\item In \cite[Algorithm 1]{VIN} (VM Alg. 1 for short) we take $\theta=0.25,~\epsilon_n=\frac{1}{(n+1)^4}$ and $\lambda=0.1$.
\item In \cite[Algorithm 4.1]{Hoai}, (PNN Alg. 4.1 for short) we take $\varphi=\frac{\sqrt{5}+1}{2}$ and $\lambda_n=\frac{1}{100(n+1)}$. 
\item  In \cite[Algorithm 3.1]{Yang} (YL Alg. 3.1 for short), we take $\mu=0.5$, $\delta=0.1, ~\alpha=0.9,~ \theta=0.1$, and $\lambda_1=0.1$. 
\item In  \cite[Algorithm 3.1]{YangOPTL},(YL36 Alg. 3.1 for short) we take $\mu=0.9$, $\lambda_0=0.4$, $\beta_n=0.5$, $\alpha_n=\frac{1}{100(n+1)}$. 
\item In \cite[Algorithm 3.1]{YinLiuYang} (YLY Alg. 3.1 for short) we take $\mu=0.1$, $\theta=0.1, ~\lambda_0=\lambda_1=0.5$.  
\end{itemize}
Then consider the following cases for the  initial points:

\textbf{Case I}:  $ y_{-1}(t)= 1-0.5t+|t-0.5|, ~~w_{-2}(t)=2\sin(t+1)$ and $y_0(t)=y_{-1}(t)$ and $w_{-1}(t)=w_{-2}(t).$

\textbf{Case II}:  $  y_{-1}(t)=t^2+1, ~~w_{-2}(t)=2\sin(t+1)$ and $y_0(t)=y_{-1}(t)$ and $w_{-1}(t)=w_{-2}(t).$

\textbf{Case III}: $ y_{-1}(t)=\exp(t)+2t-1, ~~w_{-2}(t)=2\sin(t+1)$ and $y_0(t)=y_{-1}(t)$ and $w_{-1}(t)=w_{-2}(t).$

\textbf{Case IV}: $ y_{-1}(t)=\sin(2t+1)+5, ~~w_{-2}(t)=2\sin(t+1)$ and $y_0(t)=y_{-1}(t)$ and $w_{-1}(t)=w_{-2}(t).$\\

\noindent The iteration process is terminated when $n=5001$ or $E_n=\|x_{n+1}-x_n\|>10^{-6}$. The output of the experiments is presented in Table \ref{T3} and Figure \ref{F4}.

\begin{table}[H]
\centering
\caption{Numerical performance of all algorithms in Example \ref{EX3}}
\begin{tabular}{lcccccccc}
\toprule
\multicolumn{1}{c}{\multirow{2}[4]{*}{Algorithms}} & \multicolumn{2}{c}{Case I} & \multicolumn{2}{c}{Case II} & \multicolumn{2}{c}{Case III} & \multicolumn{2}{c}{Case IV} \\
\cmidrule{2-9}          & Iter.    & Time (s)   & Iter.     & Time (s)   & Iter.     & Time (s)   & Iter.     & Time (s) \\
\midrule
Algorithm \ref{ALG} & 12    & 2.8606    & 12   & 0.9293  & 13   & 1.0405    & 14   & 0.9905 \\  
VM Alg. 1           & 133   & 28.5847   & 133  & 309.0428& 134  & 218.3022  & 132  & 131.3078 \\
PNN Alg. 4.1        & 95    & 9.3148    & 95   & 104.5991& 96   & 60.0386   & 95   & 54.7860 \\
YL Alg. 3.1         & 136   & 295.6381  & 136  & 269.1944& 136  & 272.3444  & 136  & 277.2651 \\
YL36 Alg. 3.1       & 68    & 4.6792    & 68   & 4.6675  & 68   & 4.7402    & 69   & 4.7356  \\ 
YLY Alg. 3.1        & 54    & 45.9797   & 57   & 56.9187 & 59   & 99.7894   & 96   & 60.2337 \\ 
\bottomrule
\end{tabular} \label{T3}
\end{table}	

\hfill

\begin{figure}[H]
\begin{center}
\includegraphics[width=8.0cm,height=10cm]{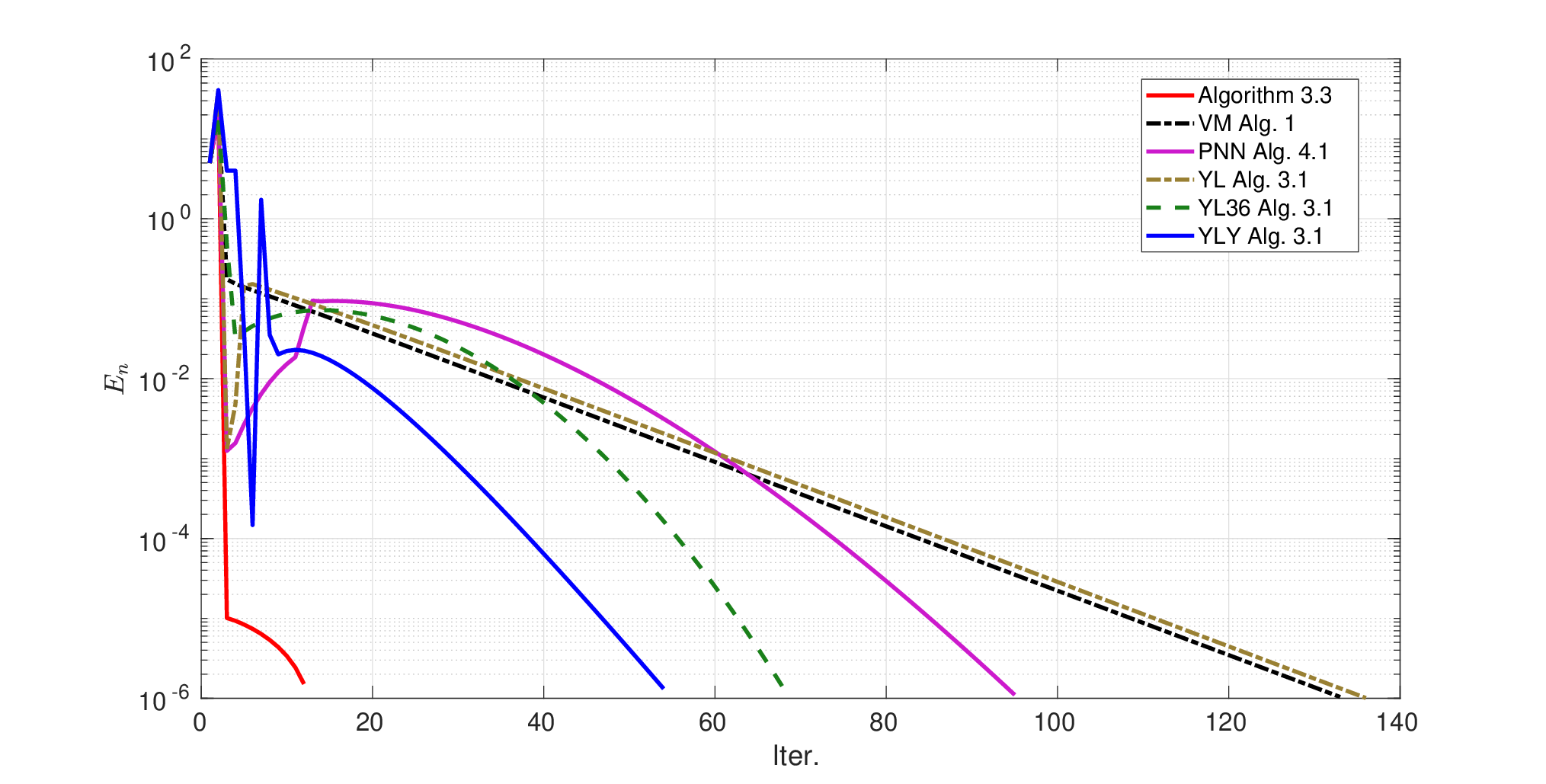} %
\includegraphics[width=8.0cm,height=10cm]{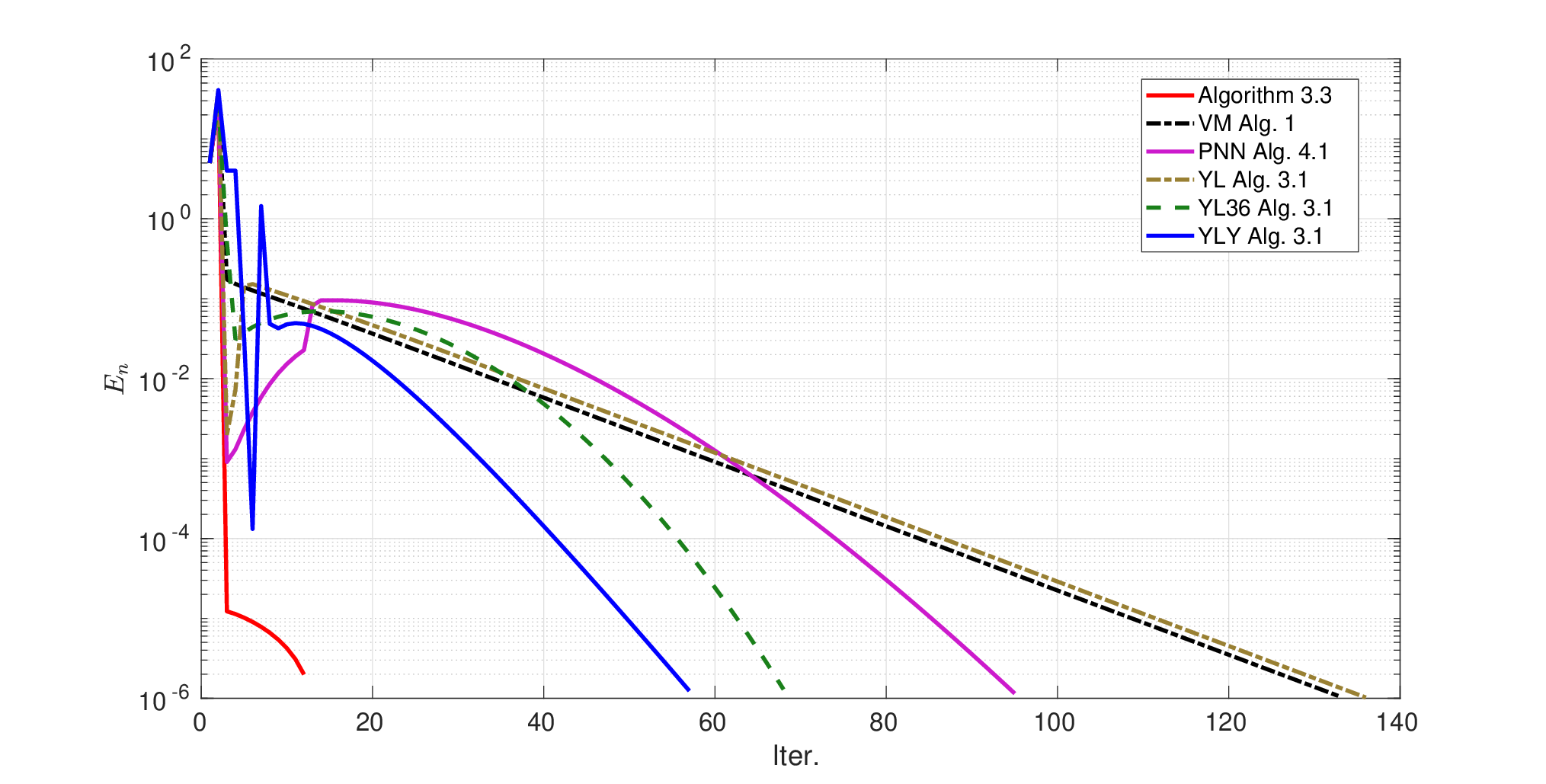} %
\includegraphics[width=8.0cm,height=10cm]{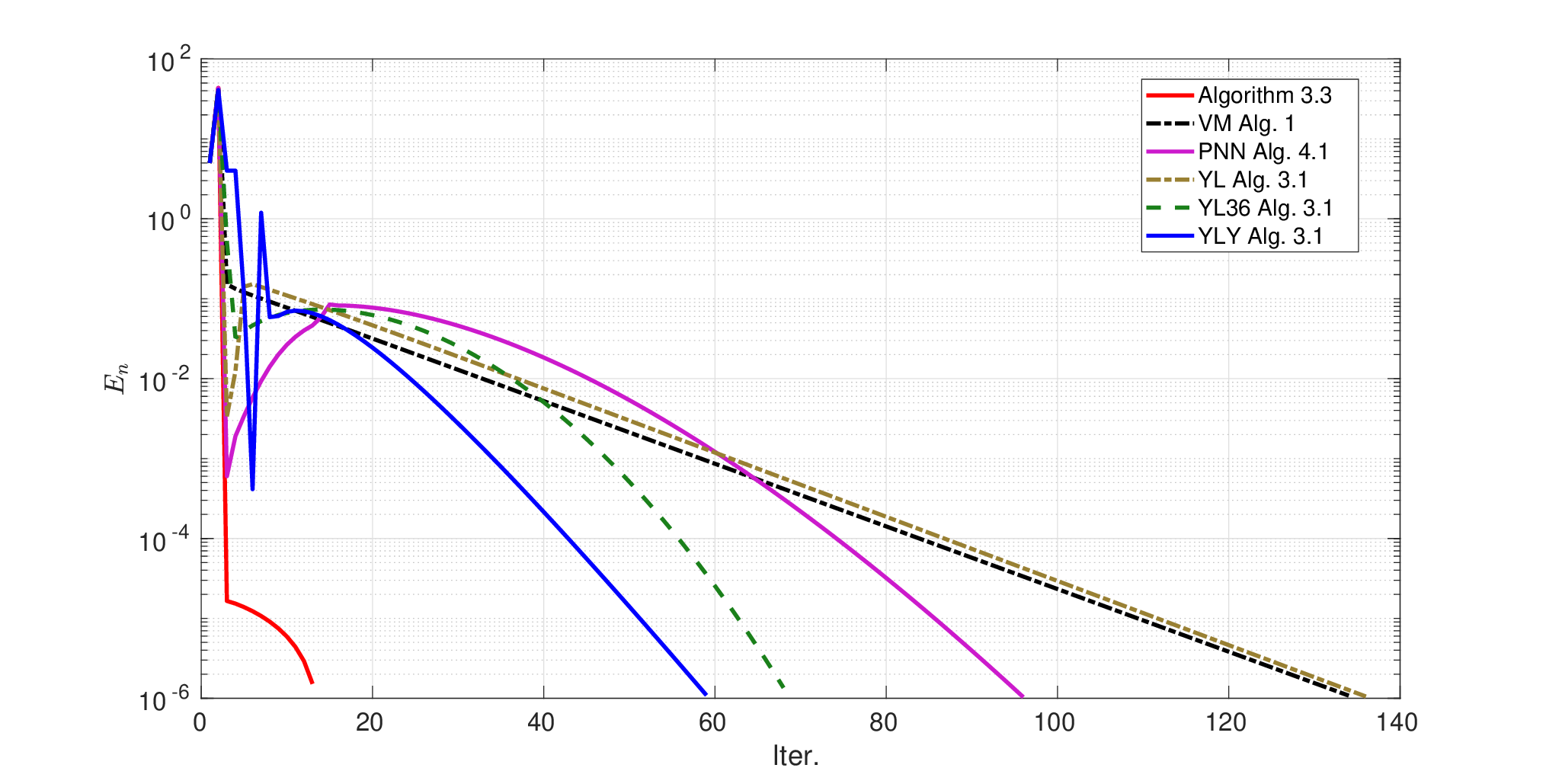}
\includegraphics[width=8.0cm,height=10cm]{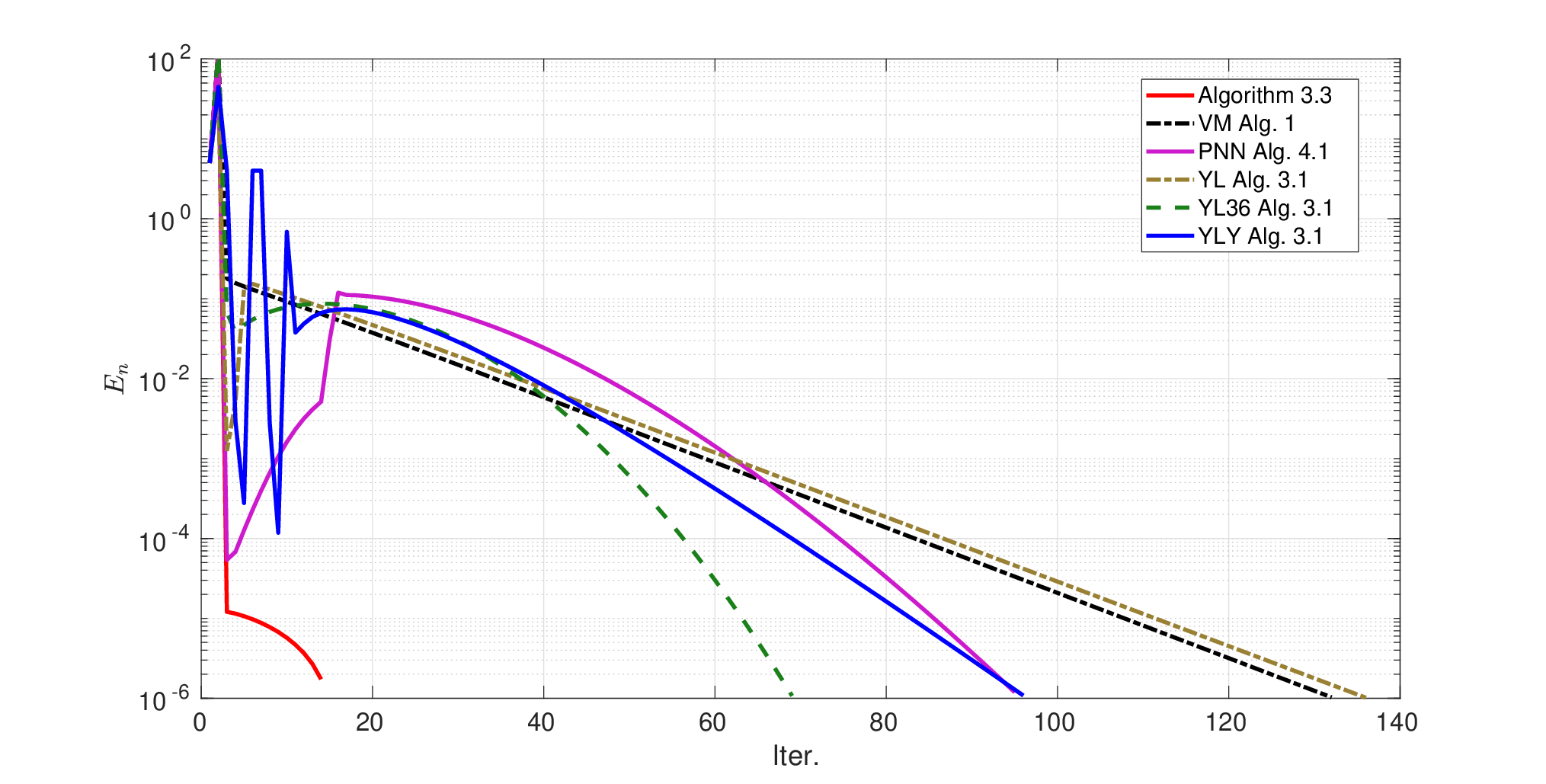}
\end{center}
\caption{Numerical results for Example \ref{EX3} for Cases I-IV, respectively} \label{F4}
\end{figure}

\end{exm}

\section{Conclusion}
\noindent In this work, we introduced a subgradient  extragradient method for solving an equilibrium problem in a Hilbert space. One thing that distinguishes our method from others is the incorporation of inertial and correction terms techniques. We obtained a weak convergence result and a linear convergence rate under standard conditions. Through some numerical examples, we were able to justify that the inclusion of these accelerating terms enhances the numerical convergence speed of the subgradient extragradient algorithm, thus outperforming the methods in \cite[Algorithm 1]{VIN}, \cite[Algorithm 4.1]{Hoai},  \cite[Algorithm 3.1]{YangOPTL}, \cite[Algorithm 3.1]{Yang} and \cite[Algorithm 3.1]{YinLiuYang}. In our future research, it is our plan to investigate a three-operator monotone inclusion problem by taking advantage of the inertial and correction terms.

\section{Declaration}
\subsection{Acknowledgement} The authors appreciate the support provided by their institutions.

\subsection{Funding}  No funding received.

\subsection{Use of AI}
The authors declare that they did not use AI to generate any part of the paper.
\subsection{Availability of data and material}  Not applicable.
\subsection{Competing interests}
The authors declare that they have no competing interests.
\subsection{Authors' contributions}
 All authors worked equally on the results and approved the final manuscript.

\end{document}